\documentclass[a4paper,11pt]{amsart}

\title[Deep--BSDE schemes for nonlocal PDEs]{ A Deep BSDE approximation of nonlinear integro-PDEs with unbounded nonlocal operators}

\usepackage{amsmath,amsthm}
\usepackage{geometry}
\usepackage[colorlinks,breaklinks=true]{hyperref}
\usepackage{graphics}
\usepackage{latexsym}
\usepackage{amsfonts}
\usepackage{stmaryrd}
\usepackage[shortlabels]{enumitem}
\usepackage{algorithm}
\usepackage{algorithmic}
\usepackage{color}

\theoremstyle{plain}
\newtheorem{theorem}{Theorem}[section]
\newtheorem{corollary}[theorem]{Corollary}
\newtheorem{lemma}[theorem]{Lemma}
\theoremstyle{definition}
\newtheorem{definition}[theorem]{Definition}
\theoremstyle{remark}
\newtheorem{remark}[theorem]{Remark}

\numberwithin{equation}{section}
\chardef\bslash=`\\ 
\newcommand{\ntt}{\normalfont\ttfamily}
\newcommand{\cn}[1]{{\protect\ntt\bslash#1}}
\newcommand{\pkg}[1]{{\protect\ntt#1}}
\newcommand{\fn}[1]{{\protect\ntt#1}}
\newcommand{\env}[1]{{\protect\ntt#1}}
\newcommand{\ee}{\normalcolor}

	\newcommand{\integ}[2]{\displaystyle \int_{#1}^{#2}}

	\renewcommand{\sectionmark}[1]{}

\newcommand{\eval}[2][\right]{\relax
	\ifx#1\right\relax \left.\fi#2#1\rvert}

\author[E. R. Jakobsen]{Espen R. Jakobsen}
\address{\parbox{.8\linewidth}{{\normalfont\textbf{E.~R.~Jakobsen}}\medskip\\
Department of Mathematical Sciences, NTNU,\\
7491 Trondheim, Norway\medskip}}
\email{espen.jakobsen@ntnu.no}

\author[S. Mazid]{Sehail Mazid}
\address{\parbox{.8\linewidth}{\normalfont\textbf{S. Mazid}\medskip\\Department of Mathematics, Faculty of Sciences, Ibn Zohr University.\\Agadir, Morocco. \medskip}} \email{s.mazid@uiz.ac.ma}
\date{\today}

\subjclass[2020]{35D40, 35K55, 35K65, 35Q84, 35Q89, 35R11, 47D07, 49L, 49N80, 60G51}	
\keywords{Deep BSDE methods, FBSDEs, machine learning methods, numerical methods, convergence, fully-nonlinear PDEs, degenerate PDEs, nonlocal PDEs, L\'evy processes}
		
\begin{document}

\maketitle

\begin{abstract}
Machine learning for partial differential equations (PDEs)
is a hot topic. In this paper we introduce and analyse a Deep BSDE scheme for
nonlinear integro-PDEs with unbounded nonlocal operators -- problems arising in e.g.~stochastic control and games involving infinite activity jump-processes. The scheme is based on a stochastic forward-backward SDE representation of the solution of the PDE and (i) approximation of small jumps by a Gaussian process, (ii) simulation of the forward part, and (iii) a neural net regression for the backward part. Unlike grid-based schemes, it does not suffer from the curse of dimensionality and is therefore suitable for high dimensional problems. The scheme is designed to be convergent even in the infinite activity/unbounded nonlocal operator case. A full convergence analysis is given and constitutes the main part of the paper.
\end{abstract}

\section{Introduction}

 In this paper we introduce and analyze a Deep BSDE scheme for
nonlinear integro-partial differential equation (integro-PDEs) with
unbounded nonlocal operators. The main result is a series of convergence results for the method. The PDEs are of the form
	\begin{equation}\label{PIDE}
\left\{
\begin{array}{l}
\displaystyle\frac{\partial u}{\partial t}(t,x)+\tilde{\mathcal{L}}[u](t,x)+f(t,x,u,\sigma\nabla u,\mathcal{B}[u])=0, \;\text{for}\;  (t,x)\in[0,T)\times\mathbb{R}^q \\  \\ u(T,x)=g(x), \;\text{for}\; x\in\mathbb{R}^q,
\end{array}
\right.
\end{equation}	
where $\tilde{\mathcal{L}}$ is  a (degenerate) elliptic 
 integral-differential operator, 
\begin{equation}\label{L}
\tilde{\mathcal{L}}u=b\cdot D_xu+\frac{1}{2}Tr(\sigma\sigma^TD_x^2u)+\mathcal{J}u,
\end{equation}
and $\mathcal{J}$ and $\mathcal{B}$ are (possibly unbounded) fractional/nonlocal operators
$$\mathcal{J} u(t,x)=\displaystyle\int_{E}\big[u(t,x+\beta(x,e))-u(t,x)-D_xu(t,x)\beta(x,e) \big]\nu(de),$$
$$\mathcal{B}[u](t,x)=\int_{E}\big[u(t,x+\beta(x,e))-u(t,x)\big]\gamma(e)\nu(de),$$
where $E = \mathbb{R}^q\setminus\{0\}$. 
 Our assumptions cover the generators $\tilde{\mathcal{L}}$ of all
$L^2$ Levy processes and Levy-driven stochastic differential
equations (SDEs, see Section \ref{sec:FBSDE}), including
processes with infinite number of jumps per time interval
where $\mathcal{J}$ and $\mathcal{B}$ become unbounded
operators. 
The tempered stable Levy processes \cite{Sa:book} 
and most of the popular processes used in Finance like the CGMY
processes \cite{CT:book} are included.
	
Equation \eqref{PIDE} describe phenomena in nature, engineering, economics, and finance. It appears e.g. as Bellman or Isaacs equation in control and game theory \cite{FS:book,Fr:book,CT:book,CDFMV17} -- and in the rapidly developing field of mean field games \cite{ACDPS20,[EJ18]}. Well-posedness can be proven with viscosity solution methods \cite{[BBP97],FS:book}, a generalised solution theory that can handle both nonlinear problems and degenerate equations with non-differentiable solutions.
Except in rare cases, this equation do not admit closed form solutions and therefore have to be solved numerically. Classical approaches like finite difference, semi-Lagrangian, and finite element methods (see e.g. \cite{[DDGGTZ20],Ta12}) are deterministic and work well in dimensions $q = 1, 2$ and $3$. But since they are grid based, their complexity grows exponentially in the dimension $q$. This is called the curse of dimensionality. 
A classical way to overcome this problem for linear PDEs, is to use Monte Carlo based simulation methods \cite{Gl:book,MT21} and  Feynman-Kac representations of $u$ as an expectation of the solution of a SDE. 

Pardoux and Peng gave representation formulas for solutions of a class of nonlinear PDEs in terms of forward backward systems of SDEs (FBSDEs), see e.g. \cite{PR14}. For local problems these representation formulas have been exploited to create Monte Carlo based simulation methods, see e.g. \cite{[BT04],[Z04],[GLW05],MT21}. 
 But costly computations of conditional expectations limit these methods to dimensions lower than 8 \cite{[BT04]}.
A new type of schemes was then developed in \cite{[EHJ17]}, capable of solving 
numerically (local) PDEs in 100 dimensions. This so-called Deep BSDE scheme involves neural net approximations and avoids computations of conditional expectations.  
 It is based on Euler discretisation  in time, simulating the forward part, and learning the 
backward part through a global in time minimisation of 
 the approximation error at the terminal time. 
This type of schemes have then been developed in many papers, see e.g. \cite{HPBL21,BBCJN21,GLV20,HPW20,[SS18],[HL18],TTY22}. In \cite{HPW20} it was observed that the methods of \cite{[EHJ17]} could be stuck in poor local minima or diverge during the optimization/learning stage.
To overcome this problem, they formulated new methods based on a sequence of local in time optimizations -- each with a dramatically reduced search space.


In this paper we will construct an extension of Deep BSDE methods to nonlocal PDEs \eqref{PIDE} and the corresponding FBSDEs which are driven by non-Gaussian Levy processes. In particular, we have chosen to extend the method of \cite{[GPW20]}, 
but extensions of other methods (including \cite{[EHJ17]}) could also be handled following the ideas of this paper.
%
%
%
%
The FBSDE repersentation for \eqref{PIDE} is known from  \cite{[BBP97]} and takes the form
\begin{align}
	\label{JFSDE}
	X_t&=x+\integ{0}{t}b(X_{s})ds+\integ{0}{t}\sigma(X_s)dW_{s}+\int_{0}^{t}\int_{E} \beta(X_{s^-},e)\tilde{\mu}(de,ds),\\
\label{JBSDE}
Y_{t}&=g(X_T)+\integ{t}{T}f(s,X_s,Y_s,Z_s,\Gamma_s)ds-\integ{t}{T}Z_sdW_s-\integ{t}{T}\int_{E}U_s(e)\tilde{\mu}(de,ds),
	\end{align}
	where $\Gamma=\int_{E}U(e)\gamma(e)\nu(de)$ and the stochastic integrals and assumptions are explained in Section \ref{sec:FBSDE}. If $u$ is a sufficiently nice solution of \eqref{PIDE}, then by the It\^{o} formula we find that $$Y_t=u(t,X_t),\qquad Z_t=\sigma^T\nabla u(t,X_t), \qquad
	U_t=u(t,X_{t^-}+\beta(X_{t^-},.))-u(t,X_{t^-}).$$
	 At $t=0$ $X_0=x$, and $Y_0=u(0,x)$ (along with $Z_0$ and $U_0$) are deterministic. In other words, since the $T$ and $x$ are arbitrary:
	 \medskip

	  {\it \quad We can find the solution $u$ of \eqref{PIDE} (and its "derivatives") at an arbitrary point $(0,x)$
	  
	  \quad by solving the FBSDE \eqref{JFSDE}-\eqref{JBSDE} for $Y_0$ ($Z_0$ and $U_0$).}
	\medskip 

\noindent We refer again to Section \ref{sec:FBSDE} for precise
statements and a detailed discussion. 

In \cite{[BE08]} the authors
introduced Monte Carlo based simulations methods to solve
\eqref{JFSDE}--\eqref{JBSDE}, but again conditional expectations need to be computed. 
 Even though a number of methods have been designed for this, e.g. tree based methods
\cite{[T21],[MPM22]}, Fourier methods \cite{[RO15],[HRC16],[RT04]}, and least-squares Monte Carlo
methods \cite{[BT04],[FTW11], [CMT10]}, 
high dimensional problems are still out of reach. Very recently simulation methods  
for high dimensional
nonlocal problems were developed in
\cite{[ABDW24],[C20],[GPP22],[RK22],WLGF23}. In \cite{[C20]} Levy measures are finite and nonlocal operators bounded, which 
simplifies the formulation and arguments. There seems to be a mistake in the construction of
the method and some steps in the convergence proof where 
$L^2$-projections and orthogonality are required. In \cite{[GPP22],[ABDW24]} the
authors propose extensions of the Deep BSDE solver of
\cite{[EHJ17]} to FBSDEs with jumps, but they have yet to show convergence
of their schemes. An extension of the
deep splitting scheme of \cite{BBCJN21} to nonlocal PDEs is  proposed in \cite{[RK22]}. Here
Levy measures are finite and nonlocal operators bounded, approximate gradients are computed by differentiating neural
network functions (can lead to reduceed accuracy \cite{HPW20}), and the analysis is restricted to 
$C^1$
neural networks (excluding e.g. ReLu activation functions). Finally, in \cite{WLGF23} the authors
propose a Deep BSDE method for coupled FBSDEs
with jumps. They prove that 
simulation errors are 
bounded by 
the value of the objective function, but not the
existence of (near) optimal neural networks which is needed for the
full convergence 
without the curse of dimensionality. 

 In this paper we propose a Deep BSDE scheme for (\ref{PIDE}) that can accomodate singular Levy measures and unbounded nonlocal operators $\tilde{\mathcal L}$ in \eqref{L}. This means that the driving Poisson noise of the FBSDE \eqref{JFSDE}--\eqref{JBSDE} may have infinite activity (infinitely many jumps per time interval). Such processes can be approximated by finite activity compound Poisson processes obtained by truncating/removing small jumps (smaller than $\epsilon>0$). An improved approximation can be obtained by approximating the small jumps by a suitably chosen ($\epsilon$ depending) Gaussian (It\^{o}) process \cite{[AR01]}, and this is by now one standard way of simulating such processes \cite{PB10,CT:book}. Our method can then be described in the following way:
\begin{itemize}
\item[(i)] In the FBSDEs we approximate small jumps by a Gaussian process (cf. \eqref{FSDE2}--\eqref{FBSDE2}).\smallskip
\item[(ii)] We discretize the resulting FBSDEs in time and their Poisson integrals in space. The result is an Euler-Maruyama type of discretisation in time for the SDE and a quadrature approximation of the BSDE (cf. \eqref{DisSDE}, \eqref{DisBSDE}).\smallskip
\item[(iii)] We simulate the forward SDE part, and compute the BSDE part through a deep neural net regression based on local in time optimisation (cf. Algorithm \ref{ALGO1}). A extension of the deep backward multi-step scheme \cite{[GPW20]} is used for the BSDE part.
\end{itemize}

Our method can be seen as a
nonlocal extension of \cite{HPW20,[GPW20]}, and as a Deep BSDE extension of the simulation method for nonlocal problems derived in \cite{[BE08]}. Note however some special features of our method which are not common in previous work on nonlocal models (including \cite{[BE08]}): 
\begin{itemize}
\item[(a)] We consider also infinite activity jump processes, where the Levy measure is singular and nonlocal operators are unbounded. \smallskip
\item[(b)] We do this in a natural way via approximation of small jumps by Gaussian processes.\smallskip
\item[(c)] We compute a pointwise (neural net) approximation also of $U$ in \eqref{JBSDE}, as opposed to other papers that compute integrals of $U$ or indirectly compute $U$ from $Y$. \smallskip
\item[(d)] We discretise the stochastic integral accounting for jumps in the FBSDEs (the $\tilde\mu$-integral) also in ($e$-)space. This is needed to be able to compute $U$ in part (c).\smallskip
\item[(e)] We give a complete convergence analysis of the scheme. 
\end{itemize}
Our scheme is designed so that good enough projections can be obtained for the convergence proof to work (see below). New difficulties arise in the presence of jumps, and the mistake of \cite{[C20]} is related to this. 
We solve this problem by introducing a further discretisation mentioned in (d) and build the good projections from the approximation of $U$ in (c).


The main results of this paper are a series of convergence results for the scheme, including weak and strong rates of convergence. This analysis is done in two steps, considering discretisation (i) first and then discretisations (ii)--(iii). For step (i) we recall weak convergence results (as $\epsilon\to0$) on the level PDEs obtained in \cite{[JKC08]} via viscosity solution techniques. We also give strong $L^2$-rates of convergence in line with \cite{[A13]}. These results are obtained by a probabilistic It\^{o}-calculus based approach, but there are some problems with the proof that we correct here (Appendix \ref{app:A}). Note that the weak convergence result of \cite{[JKC08]} gives better rates.

In the second step we analyse steps (ii)--(iii). In order to combine this analysis with the results from step (i), we need to obtain precise bounds that are uniform in $\epsilon$. Since the limit $\epsilon\to 0$ can be seen as singular perturbation limit, a precise and careful analysis is needed, an analysis that can only work with a carefully chosen discretisation of the jump part of the FBSDEs.
Following  \cite{[HL18],HPW20}, we show convergence in $L^2$ by adapting some of the machinery of \cite{[BT04]}. To do that we derive an approximation which is intermediate between the FBSDEs (cf. \eqref{FSDE2}--\eqref{FBSDE2}) and our Deep BSDE approximation (Algorithm \ref{ALGO1}). 
This scheme relies on the discrete in time and jump FBSDEs and carefully defined projections and conditional expectations. It shares features with the approximations in \cite{[BE08],[BT04]}, and is a direct extension of the intermediate scheme of \cite{HPW20}. 
As opposed to \cite{[BE08]}, we consider unbounded Levy measures and need an additional subdivision in ($e$-)space and spatial approximation of the jumps terms, that allow us to obtain a pointwise 
approximation of $U$ in the BSDE (cf.~\eqref{FBSDE2}). In \cite{[BE08]} only integrals of $U$ appear. The good projections can be defined from this approximation, but only in regions away from small jumps. 
Not only is the truncation of small jumps a practical way of simulation infinite activity processes and integral-PDEs with unbounded nonlocal operators, but it also allows us to construct the first complete convergence analysis for a Deep BSDE scheme in this setting. Finally we remark that the use a multi-step method and ideas from \cite{[GPW20]} allow us to get error estimates that are better behaved with respect to the number of time steps than e.g. in \cite{[HL18],HPW20}.

The remainder of this article is organized as follows: In section \ref{sec:pde} we state conditions and give well-posedness 
results for 
the integro-PDE \eqref{PIDE}. 
We 
discuss on the PDE-level how to approximate \eqref{PIDE} by truncating small jumps and adding a compensating Brownian motion/It\^{o} integral. 
In section \ref{sec:FBSDE0} we introduce the associated FBSDEs with jumps and recall well-posedness results and how this system is related to \eqref{PIDE}. We introduce the approximation by truncating small jumps and compensating by a Gaussian term, and give an $L^2$ convergence result for this approximation. In section \ref{sec:scheme} we propose our new Deep (multi-step) BSDE discretisation of the FBSDEs with jumps, 
introduce the intermediate scheme, 
and state the convergence results -- the main results of this paper. Section \ref{sec:pf} is devoted to the proofs of our results. Some proofs are given in the appendices in order not to disrupt the flow of the paper.

	\section{Nonlocal nonlinear partial differential equations.}
\label{sec:pde}

In this section we state the (standard) assumptions and give well-posedness and approximation results for the nonlocal nonlinear PDE \eqref{PIDE}.
\begin{itemize}
    \item[$(\mathbf{H1})$] (i) The functions $b:\mathbb{R}^q\to\mathbb{R}^{q}$, $\sigma:\mathbb{R}^q\to\mathbb{R}^{q\times d}$, and $\beta:\mathbb{R}^q\times E\to\mathbb{R}^{q}$
are continuous and there exists $K>0$ such that $ \forall x,x',e\in\mathbb{R}^q,$
\begin{equation}
\frac{|b(x)-b(x')|}{|x-x'|}+\frac{|\sigma(x)-\sigma(x')|}{|x-x'|}+\frac{|\beta(x,e)-\beta(x',e)|}{|x-x'||e|}\leq K,
\end{equation}
\begin{equation}\label{Con-Dbeta}
|\beta(x,e)|\leq K(1+|x|)(1\wedge|e|), \quad\text{and}\quad|D_e\beta(x,0)|\leq K.
\end{equation} (ii) $\nu(de)$ is a $\sigma$-finite nonnegative Borel measure on $E$ satisfying
\begin{equation}\label{levy}
\int_{E}|e|^2\nu(de)<+\infty.
\end{equation}
\item[$(\mathbf{H2})$] (i) There exists $C_f > 0$ such that the driver $f:[0,T]\times\mathbb{R}^q\times\mathbb{R} \times\mathbb{R}^{d}\times\mathbb{R}^q\to\mathbb{R}$ satisfies:
$$\frac{|f(t_2,x_2,y_2,z_2,p_2)-f(t_1,x_1,y_1,z_1,p_1)|}{|t_2-t_1|^\frac{1}{2}+|x_2-x_1|+|y_2-y_1|+|z_2-z_1|+|p_2-p_1|}\leq C_f, $$
for all $(t_1,x_1,y_1,z_1,p_1)$ and $(t_2,x_2,y_2,z_2,p_2)\in[0,T]\times\mathbb{R}^q\times\mathbb{R}\times\mathbb{R}^d\times\mathbb{R}^q$. Moreover,
$$\sup_{0\leq t\leq T}|f(t,0,0,0,0)|<\infty.$$
(ii) The function $\gamma:E\to
\mathbb{R}$ is continuous and there exists $C_\gamma, K> 0$ such that
\begin{equation}\label{Condi-gamma}
0\leq \gamma(e)\leq C_\gamma(1\wedge |e|) \quad\text{and}\quad|D_e\gamma(0)|\leq K.
\end{equation}
(iii) There exists $C_g> 0$ such that the function $g:\mathbb{R}^q\to \mathbb{R}$ satisfies, $\forall x,x'\in\mathbb{R}^q$
\begin{equation}
|g(x)-g(x')|\leq C_g|x-x'|.
\end{equation}
\end{itemize}

\begin{remark}\label{rem:as}
(a) Assumption (\ref{levy}) ensures
that underlying FBSDE \eqref{JFSDE}--\eqref{JBSDE} have 2 finite moments, so we can use $L^2$ theory and the proofs simplifies. 
\smallskip

\noindent (b)   All (PDE) results in this section holds if we replace (\ref{levy}) by the general Levy integrability condition $\int_E|e|^2\wedge 1\  \nu(de)<\infty$.
\smallskip

\noindent(c) By the above assumptions, $\beta(x,0)=0=\gamma(0).$
\smallskip

\noindent(d) We will sometimes split the nonlocal operator into singular and nonsingular parts: $$\mathcal{J}u(t,x)=\mathcal{J}_1^\epsilon u(t,x)+\mathcal{J}_2^\epsilon u(t,x)=\int_{|e|\leq \epsilon}\dots+\int_{|e|> \epsilon}\dots.$$
\end{remark}
By \cite{[BBP97]}, we have the following wellposedness result.
\begin{theorem}\label{thm:lip}
Assume (\textbf{H1}) and (\textbf{H2}). Then there exists  a unique viscosity solution $u$  of (\ref{PIDE}). Moreover, there exists $C> 0$ such that for $x,x'\in\mathbb{R}^q$ and $t\in[0,T)$,
\begin{equation}
|u(t,x)-u(t,x')|\leq C|x-x'|.
\end{equation}
\end{theorem}
This viscosity solution can be smooth under additional assumptions, for example by assuming there is $k\in\mathbb{N}$ such that the following assumption holds:
\smallskip
\begin{itemize}
\item[($\mathbf{H1'})_k$] (i) For every $R>0$ there is $C_R>0,$ such that for $t\in[0,T), x\in\mathbb {R}^q, u\in[-R,R],|\alpha|\leq k,$
$$ |D^\alpha_{x,u,p,q}f(t,x,u,p,q)|\leq C_R.$$
(ii) The terminal condition $g\in W^{k,\infty}(\mathbb{R}^q).$\smallskip

\noindent(iii) $\beta$ and $\sigma$ do not depend on $x$.
\end{itemize}
\smallskip

\noindent Then we have the following result from \cite{[EJ18]}.
\begin{theorem}
Assume $k \geq 2$, (\textbf{H1}), (\textbf{H2}) and (\textbf{H1}')$_k$ hold. Then (\ref{PIDE}) has a unique classical solution $u$ satisfying
$$
\partial_{t} u, u, D u, \cdots, D^{k} u \in C_{b}\big((0, T) \times \mathbb{R}^{d}\big)
	$$
	with
	$$
	\left\|\partial_{t} u\right\|_{L^{\infty}}+\|u\|_{L^{\infty}}+\|D u\|_{L^{\infty}}+\ldots+\|D^{k} u\|_{L^{\infty}} \leq C,
	$$
	where $C$ is a constant depending only on $T, q, k,$ and the constants in (\textbf{H1}')$_k$.
\end{theorem}
\smallskip

We now discuss how to approximate (\ref{PIDE}) by a problem with nonsingular nonlocal operator $\mathcal J$, or equivalently, with bounded Levy measure $\nu$. A naive idea is to remove the singular part of $\nu$,
and thereby the small jumps of the driving process: Replacing $\nu(de)$ by $\nu(de) 1_{|e|>\epsilon}$ and $\mathcal J$ in (\ref{PIDE}) by $\mathcal{J}_2^\epsilon$ defined in Remark \ref{rem:as} (d). The resulting approximation of \eqref{PIDE} converges as $\epsilon\to 0$ \cite{[JKC08]}, but 
slowly when small jumps are important, as for all $\alpha$-stable like processes and in most commonly used models in finance 
\cite{CT:book}.
 A better approximation is obtained by approximating small jumps by 
 an It\^{o} integral \cite{[AR01]}. On the level of operators and PDEs, this means that we replace $\mathcal{J}_1^\epsilon$-term
in (\ref{PIDE}) by the local diffusion term \cite{CJ22} (see also \cite{CT:book,[JKC08]}) $\frac
1{2}Tr[\sigma_\epsilon\sigma_\epsilon^T D^2_xu]$ where 
\begin{align}\label{def-Sigma}
\sigma_\epsilon = D_e\beta(x,0) \Sigma^{\frac{1}{2}}_\epsilon\qquad\text{and}\qquad \Sigma_\epsilon=\int_{|e|\leq \epsilon}ee^T\nu(de).
\end{align}
The resulting two approximations of (\ref{PIDE}) then becomes
\begin{equation}
\label{PIDE2}\left\{
\begin{array}{l}
\displaystyle\frac{\partial u^\epsilon}{\partial t}(t,x)+\tilde{\mathcal{L}}^\epsilon[u^\epsilon](t,x)+f(t,x,u^\epsilon,\sigma\nabla u^\epsilon,\mathcal{B}^\epsilon[u^\epsilon])=0 \quad\text{in}\quad [0,T)\times\mathbb{R}^q, \\  \\ u^\epsilon(T,x)=g(x),
\end{array}
\right.
\end{equation}
where  for $\zeta\in\{0,1\}$,
\begin{align*}
\tilde{\mathcal{L}}^\epsilon[u](t,x)= &\ bD_xu(t,x)+
\frac{1}{2}Tr\big(\big[\sigma\sigma^T
+\zeta\sigma_\epsilon\sigma_\epsilon^T
\big]D_x^2u(t,x)\big)
\\ &+\int_{E^\epsilon}\Big(u(t,x+\beta(x,e))-u(t,x)-D_xu(t,x)\beta(x,e) \Big)\nu(de),\\[0.2cm] \mathcal{B}^\epsilon[u](t,x)=&\int_{E^\epsilon}\big[u(t,x+\beta(x,e))-u(t,x)\big]\gamma(e)\nu(de)+
\zeta\sigma_\epsilon^TD\gamma(0)D_x u(t,x),
\end{align*}
and where
$ E^\epsilon:=\{e\in E : |e|> \epsilon\}$.  Note that $\zeta=0$ means no additional diffusion term.

\medskip Next, we need to be more precise about the behavior of the Levy measure near the origin:\\ \\
\medskip $(\mathbf{H3})$ There is a density $m : E \to [0,\infty)$ and an $\alpha\in(0,2)$, such that
$\nu(de)=m(e)de$ and
\begin{equation}
m(e)\leq C\frac{1}{|e|^{q+\alpha}}\qquad \text{for}\quad |e|\leq 1.
\end{equation}
This condition  implies that the (pseudo) differential operator $\mathcal{L}$ has order at most $\alpha$. 
We emphasize that the conditions in $(\mathbf{H3})$ are satisfied by all Levy processes used in the literature to model financial markets, see \cite{[RT04]}. Now we give an explicit estimate for the error committed by replacing (\ref{PIDE}) by (\ref{PIDE2}).  The proofs are similar to \cite{[JKC08]}.
	
	\begin{theorem}\label{estim-weak}
Assume $\epsilon>0$, $\zeta\in\{0,1\}$, (\textbf{H1}), (\textbf{H2}), (\textbf{H3}), and let $u$ and $u^\epsilon$ be the
		viscosity solutions of (\ref{PIDE}) and (\ref{PIDE2}) respectively.
\smallskip

\noindent a)  Then for $\epsilon > 0$ small enough, there is $C\geq0$ such that for $(t,x)\in [0,T)\times\mathbb{R}^q$,
		$$|(u-u^{\epsilon})(t,x)|\leq C(1+|x|)(\zeta\epsilon^{1-\alpha/3}+(1-\zeta)\epsilon^{1-\alpha/2}).$$
			
\noindent b) Moreover, if (\textbf{H1'}) holds for $k\geq 3$, then there is $C\geq0$ such that
		$$|u-u^{\epsilon}|_0\leq C(\zeta\epsilon^{3-\alpha}+(1-\zeta)\epsilon^{2-\alpha}).$$
	\end{theorem}
\begin{remark}
 The estimate in a) depend linearly on $x$ because coefficients there (as opposed to part b)) are allowed linear growth. 
  The rate in b) can not be improved unless $\nu$ is symmetric and $k\geq 4$. In this case the rate is $\epsilon^{4-\alpha}$.
\end{remark}
	\section {Connection  to forward-backward stochastic differential equations with jumps}\label{sec:FBSDE0}

	\subsection{Notation and Preliminaries.}\label{sec:FBSDE}
	Let $(\Omega,\mathcal{F},\{\mathcal{F}_t\}_{0\leq t\leq T},\mathbb{P})$ be a stochastic basis satisfying the usual hypotheses of completeness and right continuity, i.e., $\mathcal{F}_0$ contains all the sets of $\mathbb{P}$-measure zero and $\mathcal{F}_t = \mathcal{F}_{t^+}$. The filtration $\{\mathcal{F}_t\}_{0\leq t\leq T}$ is assumed to be generated by two mutually independent processes, a $d$-dimensional Brownian motion $W_t = (W_t^1,\dots, W_t^d )^T$ and a Poisson random measure $\mu(A,t)$ on $E\times[0,T]$. The compensator and compensated Poisson random measure associated to $\mu$ are denoted by $\lambda(de, dt) = \nu(de)dt$ and $\tilde{\mu}(de, dt) = \mu(de, dt)-\nu(de)dt$. Note that $\{\tilde{\mu}(A\times[0,t])
	\}_{0\leq t\leq T}$ is a martingale for all $A\in \mathcal{E}$.
	Given
	$s \leq t$ and $K\subset E$, we use the following spaces for stochastic processes.\medskip
	\begin{enumerate}[a)]
		\item $\mathbb{S}^2_{[s,t]}$ is the set of real valued adapted c\`{a}dl\`{a}g processes $Y$ such that
		$$\|Y\|_{\mathbb{S}^2_{[s,t]}}:=\mathbb{E}\bigg[\sup_{s\leq r\leq t}|Y_r|^2\bigg]^{\frac{1}{2}}<\infty.$$
		\item $\mathbb{H}^2_{[s,t]}$ is the set of progressively measurable $\mathbb{R}^q $-valued processes $Z$ such that
		$$\|Z\|_{\mathbb{H}^2_{[s,t]}}:=\mathbb{E}\bigg[\integ{s}{t}|Z_r|^2dr
  \bigg]^{\frac{1}{2}}<\infty.$$
		\item $\mathbb{L}^2_{\nu,[s,t],K}$
		is the set of $\mathcal{P}\otimes \mathcal{K}$ measurable maps $U : \Omega \times [0,T]\times K \to \mathbb{R}$ such that
		$$\|U\|_{\mathbb{L}^2_{\nu,[s,t],K}} :=\mathbb{E}\bigg[\integ{s}{t}\int_{K}|U_r(e)|^2\nu(de)dr\bigg]^{\frac{1}{2}}<\infty,$$
		where $\mathcal{P}$ and $\mathcal{K}$ are the $\mathcal F_t$-predictable $\sigma$-algebra on $\Omega\times[0,T]$ and Borel $\sigma$-algebra of $K$.\medskip
		\item
		$\mathcal{K}^2_{\nu,[s,t],K}$ is the set of functions $(Y,Z,U)$ in the space $\mathbb{S}^2_{[s,t]}\times\mathbb{H}^2_{[s,t]}\times \mathbb{L}^2_{\nu,[s,t],K}$ with the norm defined by
		$$\|(Y,Z,U)\|_{\mathcal{K}^2_{\nu,[s,t],K}}:=\big(\|Y\|^2_{\mathbb{S}^2_{[s,t]}}+\|Z\|^2_{\mathbb{H}^2_{[s,t]}}+\|U\|^2_{\mathbb{L}^2_{\nu,[s,t],K}}\big)^\frac{1}{2}.$$
		\item $\widetilde{\mathbb{H}}^2_{[s,t]}$  is the set 
		of 
		$\mathbb{H}^2_{[s,t]}$ 
		processes which are constant in $[s,t].$\medskip
		\item
		$\widetilde{\mathbb{L}}^2_{\nu,[s,t],K}$  is the 
		set of $\mathbb{L}^2_{\nu,[s,t],K}$ 
		processes which are constant in $[s,t]$ and $K.$
	\end{enumerate}
	
	Now recall that in $L^2$ the conditional expectation is an orthogonal projection.
	\begin{theorem}\label{teoremprojection}
		Let $\mathcal{X}\in L^2(\mathcal{F})$ and $\mathcal{G}\subset \mathcal{F}$ be a sub $\sigma$-algebra. There exists a random variable $\mathcal{Y}\in L^2(\mathcal{G})$ such that
		$$\mathbb{E}|\mathcal{X}-\mathcal{Y}|^2=\inf_{\eta\in L^2(\mathcal{G})}\mathbb{E}|\mathcal{X}-\eta|^2.$$
		The minimizer $\mathcal{Y}$ is unique and is given by $\mathcal{Y} = E[\mathcal{X}|\mathcal{G}],$ and 
		$$\mathbb{E}[(\mathcal{X}-\mathcal{Y})\mathcal{Z}]=0\qquad \text{for all}\;\; \mathcal {Z}\in L^2(\mathcal{G}).$$
	\end{theorem}
A proof can be found in e.g. \cite[Chp. 3 Thm. 14]{[K02]}.
\begin{definition}\label{def:proj}
 Let  $\pi_ {[s,t]}(Z)$ and  $\pi_ {[s,t],K}(U)$ be the orthogonal $L^2$- projections of $Z$ and $U$ into $ \widetilde{\mathbb{H}}^2_{[s,t]}$ and $ \widetilde{\mathbb{L}}^2_{\nu,[s,t],K}$, 
  \begin{align*}&\mathbb{E}\bigg[\integ{s}{t}\big(Z_r-\pi_ {[s,t]}(Z_r)\big)\mathcal{Z}dr\bigg]=0 && \text{for all}\;\; \mathcal {Z}\in  \widetilde{\mathbb{H}}^2_{[s,t]},\\
  &\mathbb{E}\bigg[\integ{s}{t}\int_{K}(U_r(e)-\pi_ {[s,t],K}(U_r))\,\mathcal{U}\,\nu(de)dr\bigg]=0 && \text{for all}\;\; \mathcal {U}\in  \widetilde{\mathbb{L}}^2_{\nu,[s,t],K}.
  \end{align*}
\end{definition}
	\begin{corollary}\label{cor:proj}
		Let $Z\in\mathbb{H}^2_{[s,t]}$ , $U\in\mathbb{L}^2_{\nu,[s,t],K}$, and $\nu(K)\ne0$, then
		\begin{equation*} \pi_ {[s,t]}(Z)=\frac{1}{t-s}\mathbb{E}\bigg[\int_{s}^{t}Z_rdr|\mathcal{F}_s\bigg]\quad\text{and}\quad\pi_ {[s,t],K}(U)=\frac{1}{(t-s)\nu(K)}\mathbb{E}\bigg[\int_{s}^{t}\int_{K}U_r(e)\nu(de)dr|\mathcal{F}_s\bigg].
		\end{equation*}
	\end{corollary}
	\begin{remark}
			Note that  $\mathbb{E}\big[\int_{s}^{t}\big|Z_r-\pi_ {[s,t]}(Z)|^2dr\big]$, $\mathbb{E}\big[\int_{s}^{t}\int_{K}\big|U_r(e)-\pi_ {[s,t],K}(U)\big|^2\nu(de)dr\big]\to 0$ as $|t-s|\to0$ and $diam(K)\to 0$.
	\end{remark}
	\subsection {Forward-backward stochastic differential equation with jumps.}\label{sec:FBSDE}
	First we consider the Forward Levy-driven SDE  on $(t,T)$, 
	\begin{equation}
	\label{FSDE}
	X_{s}=x+ \integ{t}{s}b(X_{r})\,dr+\integ{t}{s}\sigma(X_r)\,dW_{r}+\int_{t}^s\int_{E} \beta(X_{r^-},e)\,\tilde{\mu}(de,dr).
	\end{equation}
	Under assumption (\textbf{H1}) there exist a unique strong (a.s.) solution $X_s\in \mathbb{S}^2_{[t,T]}$ of (\ref{FSDE}) with
	\begin{equation} \label{LG-X}
	\mathbb{E}\big[\displaystyle\sup_{t\leq s\leq T}|X_s|^2\big]\ee\leq C(1+|x|^2),\end{equation}
	see e.g. \cite[Chp. V, Thm. 7 and 67]{[P04]}. 
	Then we consider the Backward SDE  on $(t,T)$: 
	\begin{equation}
	\label{BSDE}Y_{s}=g(X_T)+\integ{s}{T}f(s,X_s,Y_s,Z_s,\Gamma_s)\,ds-\integ{t}{T}Z_s\,dW_s-\integ{t}{T}\int_{E}U_s(e)\,\tilde{\mu}(de,ds),
	\end{equation}
	where $\Gamma=\int_{E}U(e)\gamma(e)\nu(de).$ Under assumptions (\textbf{H1}) and (\textbf{H2}), there exists a unique strong (a.s.) solution  $(Y, Z,U)\in\mathcal{K}^2_{\nu,[t,T],E}$ \ee of (\ref{BSDE}), and the (deterministic) function  $$u(t, x)=Y_{t}^{t,x},$$ 
	is a viscosity solution of the PIDE (\ref{PIDE})  (cf. \cite{[BBP97]}).  Here and below we denote the solutions by $X_s=X_s^{x,t}$, $Y_s=Y_s^{x,t}$ etc. to indicate the initial condition $X_t=x$. \ee
Moreover if $u$ is a classic solution to (\ref{PIDE}), then by Lemma \ref{u=Y} below, $Y^{t,x}_s=u(s,X_s^{t,x}),\
	Z_s^{t,x}=\sigma^T(X_s^{t,x})\nabla u(s,X^{t,x}_s)\ \text{and}\
	U_s^{t,x}=u(s,X_{s^-}^{t,x}+\beta(X_{s^-}^{t,x},.))-u(s,X_{s^-}^{t,x}).$ 
	
		Next following  (e.g.) \cite{[AR01],[A13]} we approximate these equations replacing the infinite activity jump measure by one of finite activity (truncating small jumps), possibly compensating  by adding a suitably scaled Brownian motion $\widetilde{W}$ independent of $W$ and $\mu$: For $\epsilon>0$ and $\zeta\in\{0,1\}$,
\begin{align}\label{FSDE2}
X^\epsilon_s &=x+\integ{t}{s}b(X^\epsilon_{r})\,dr+\integ{t}{s}\sigma(X^\epsilon_r)\,dW_{r}+\zeta\integ{t}{s} D_e\beta(X^\epsilon_{r^-},0)\Sigma^{\frac{1}{2}}_\epsilon \,d\widetilde{W}_{r}\\
&\qquad+\displaystyle\int_{t}^{s}\int_{E^\epsilon} \beta(X^\epsilon_{r^-},e)\,\tilde{\mu}(de,dr),\nonumber
\\
\intertext{and}
	\label{FBSDE2}
Y^\epsilon_{s}&=g(X^\epsilon_T)+\integ{s}{T}f(r,X^\epsilon_r,Y^\epsilon_r,Z^\epsilon_r,\Gamma^\epsilon_r+\zeta D_e\gamma(0)\Sigma^{\frac{1}{2}}_\epsilon L^\epsilon_s)\,dr-\zeta\integ{s}{T}L^\epsilon_r\, d\widetilde{W}_r-\integ{s}{T}Z^\epsilon_r\,dW_r\\
&\nonumber\qquad\;-\integ{s}{T}\int_{E^\epsilon}U^\epsilon_r(e)\,\tilde{\mu}(de,dr).
	\end{align}\ee
where  $\Sigma_\epsilon$ is defined in \eqref{def-Sigma} and $\Gamma^\epsilon=\int_{E^\epsilon}U^\epsilon(e)\gamma(e)\,\nu(de)$. By e.g. \cite{[BBP97]}, the FBSDE  (\ref{FSDE2})-(\ref{FBSDE2}) is well-posed under (\textbf{H1}) and (\textbf{H2}).
	\begin{theorem}\label{WP_FBSDE}
		Assume $\epsilon>0$,  $\zeta\in\{0,1\}$, (\textbf{H1}), and (\textbf{H2}). Then there exists a unique strong (a.s.) solution $(X^\epsilon,Y^\epsilon, Z^\epsilon,U^\epsilon{,L^\epsilon})\in\mathbb{S}^2_{[t,T]}\times\mathcal{K}^2_{\nu,[t,T],E}{ \times\mathbb{H}^2_{[t,T]}}$ of (\ref{FSDE2})-(\ref{FBSDE2}).
	\end{theorem}

	We also have the standard mean-square continuity result for $X^\epsilon$ and $Y^\epsilon$.
	\begin{lemma}\label{WPeps}
		Assume $\epsilon>0$,  $\zeta\in\{0,1\}$, (\textbf{H1}) and  (\textbf{H2}). Then there is $C>0$ independent of $\epsilon,\zeta$, such that for  $s<s'\leq T$,
		\begin{align}
  \label{moment-Y}
&\mathbb{E}\bigg[\sup_{r\in[t,T]}|X^\epsilon_r|^2\bigg]+\mathbb{E}\bigg[\sup_{r\in[t,T]}|Y^\epsilon_r|^2\bigg]\leq  C(1+|x|^2),\\
&\mathbb{E}\bigg[\sup_{r\in[s,s']}|X^\epsilon_r-X^\epsilon_{s'}|^2\bigg]+\mathbb{E}\bigg[\sup_{r\in[s,s']}|Y^\epsilon_r-Y^\epsilon_{s'}|^2\bigg]\leq C(1+|x|^2)|s-s'|.\label{MSC-Y}
\end{align}
 
	\end{lemma}
	\begin{proof}
	This is standard. A proof can be found in e.g. \cite[Appendix A]{[BE08]},  see also \cite{PR14,ApBook09} for similar results.  Checking the proof we find that the bound for $X^\epsilon_s$ is independent of $\epsilon$. E.g. by the Levy-It\^{o} isometry for the Poisson-random measure and (H1): 
	\begin{align*}
	\mathbb E \Big|\integ{s}{s'}\int_{E^\epsilon}\beta(X_r^-,e)\,\tilde{\mu}(de,dr)\Big|^2 &= \mathbb E \integ{s}{s'}\int_{E^\epsilon}\big|\beta(X_r^-,e)\big|^2\,\nu(de)\,dr\\
	&\leq  |s-s'| C\Big(1+ \max_{s\in[t,T]}\mathbb E|X_s|^2\Big) \int_{E^\epsilon}|e|^2\,\nu(de).
	\end{align*}
In view of Lemma \ref{u=Y} below, the bound on $Y_t^\epsilon$ then follows by the  regularity of the viscosity solution $u^\epsilon$ of \eqref{PIDE2} \cite{[BBP97],[JKC08]} (H\"older-$\frac12$ in time and Lipschitz in space).
 Alternatively, the $Y_t^\epsilon$-result follows form Lemma A.2 in \cite{{[BE08]}}, Lemma \ref{u=Y}, and spatial Lipschitz result for $u$ in Lemma \ref{thm:lip}.
	\end{proof}
	
The next lemma explains the relation between the FBSDE (\ref{FSDE2}) -(\ref{FBSDE2}) and the PIDE (\ref{PIDE2}).
	\begin{lemma}\label{u=Y} Assume $\epsilon>0$,  $\zeta\in\{0,1\}$, (\textbf{H1}), and (\textbf{H2}).
	\smallskip
	
	\noindent (a)  The function  $u^\epsilon(t, x)=Y_{t}^{\epsilon,t,x}$, $(t, x) \in[0, T] \times \mathbb{R}^{q},$ is a viscosity solution of (\ref{PIDE2}).\smallskip
	
	\noindent	(b) If $u^\epsilon$ is a classic solution of (\ref{PIDE2}), then  $(Y^\epsilon, Z^\epsilon,L^\epsilon,U^\epsilon)$  defined by
		\begin{equation}\label{relation-u-Y-1}
		\left\{
		\begin{array}{l}
		Y^{\epsilon}_s=u^\epsilon(s,X_s^{t,x,\epsilon}),\\ \\
		Z_s^{\epsilon}=\sigma^T(X_s^{t,x,\epsilon})D_x u^\epsilon(s,X^{t,x,\epsilon}_s),\\ \\L_s^{\epsilon}=(\Sigma^{\frac{1}{2}}_\epsilon)^T (D_e\beta(X^\epsilon_{s^-},0))^TD_x u^\epsilon(s,X^{t,x,\epsilon}_s),\\ \\
		U_s^{\epsilon}(e)=u^\epsilon(s,X_{s^-}^{t,x,\epsilon}+\beta(X_{s^-}^{t,x,\epsilon},e))-u^\epsilon(s,X_{s^-}^{t,x,\epsilon}),
		\end{array}
		\right.
		\end{equation}
		is a strong (a.s.) solution of the  BSDE (\ref{FBSDE2}).
	\end{lemma}

\begin{remark}\label{U-u nonsmooth}
The relations for $Y^\epsilon$ and $U^\epsilon$ in \eqref{relation-u-Y-1} also holds under the assumptions of part (a) where $u^\epsilon$ is a non-smooth (but $x$-Lipschitz) viscosity solution of \eqref{PIDE2}, see e.g. \cite{[BBP97],Ha16}. 
\end{remark}
 
\begin{proof}
(a) The proof is quite standard and almost identical to the proof of Theorem 3.4 in \cite{[BBP97]}.\smallskip

\noindent (b) Follow from direct application of It\^{o}'s formula to $u^\epsilon(t,X_t)$, see e.g. \cite{[BBP97]}.
\end{proof}
 Next we show that $(Y^\epsilon, Z^\epsilon,
U^\epsilon)$ converges strongly to $(Y, Z,U)$ and  estimate the corresponding error. \ee Note that the corresponding weak convergence results of Theorem \ref{estim-weak} give better rates.
	\begin{theorem}\label{strong convergence-app}
		Assume $\epsilon>0$, $\zeta\in\{0,1\}$, (\textbf{H1}) and (\textbf{H2}). Then there is $C>0$ independent of $\epsilon,\zeta$ such that
		\begin{equation*}
		\begin{split}
		&\mathbb{E}\Big[\sup_{s\in[t, T]}\big|X_s-X_s^\epsilon\big|^2\Big]+\mathbb{E}\Big[\displaystyle\sup_{s\in[t,T]}\big|Y_s-Y^\epsilon_s\big|^2\Big]+\mathbb{E}\Big[\integ{t}{T}\big|Z_s-Z_s^\epsilon\big|^2\,ds\Big]\\
		&\qquad+\mathbb{E}\Big|\integ{0}{T}\zeta L^\epsilon_s\,d\widetilde{W}_s-\integ{t}{T}\int_{|e|\leq\epsilon}U_s(e)\,\tilde{\mu}(de,ds)\Big|^2+\mathbb{E}\Big[\integ{t}{T}\int_{E^\epsilon}\big|U_s(e)-U^\epsilon_s(e)\big|^2\,\nu(de)ds\Big]\\[0.2cm]
&\leq C(1+|x|^2)\sigma_\epsilon^2 \qquad \text{where} \qquad \sigma_\epsilon^2=\int_{|e|\leq \epsilon}|e|^2\,\nu(de).
		\end{split}
		\end{equation*}
	\end{theorem}
\begin{proof}
A proof is given in \cite{[A13]}, but there are some problems with this proof since there $U$ does not dependent on $e$. A
correct proof is given in Appendix \ref{app:A}.\end{proof}

	\section{Deep FBSDE approximation of PIDE (\ref{PIDE2}).}\label{sec:scheme}
	\subsection{Discretization of the FBSDE}
	We define a partition/grid $\mathcal G_{\Delta t}:=\{t_i: i=0,1,\dots, N\}$ on $[0,T]$ such that $ t_0=0<t_1<\dots< t_N=T$ and  denote the time steps by
	$$\Delta t_i:=t_{i+1}-t_i\quad\text{and}\quad\Delta t:=\max_{1\leq i\leq N-1}\Delta t_i.$$
	A jump-corrected Euler type discretization $X^{\epsilon,\Delta t}
	$ of the forward SDE (\ref{FSDE2}) reads
	\begin{equation}\label{DisSDE}
	\begin{array}{ll}
	X^{\epsilon,\Delta t}_0=x,\quad X_{t_{i+1}}^{\epsilon,\Delta t}=X_{t_i}^{\epsilon,\Delta t}+b(X_{t_i}^{\epsilon,\Delta t})\Delta t_i+\sigma(X_{t_i}^{\epsilon,\Delta t})\Delta W_{t_i}+\zeta D_e\beta(X_{t_i}^{\epsilon,\Delta t},0)\Sigma_\epsilon^{\frac{1}{2}}\Delta \widetilde{W}_{t_i}\\ \\ \hspace{3.7cm}+\displaystyle\ \int_{E^\epsilon}\beta(X^{\epsilon,\Delta t}_{t_i},e)\tilde{\mu}(de,(t_i,t_{i+1}]),
	\end{array}
	\end{equation}
	where $\Delta W_{t_i}:=W_{t_{i+1}}-W_{t_i}$ and $\Delta \widetilde{W}_{t_i}:=\widetilde{W}_{t_{i+1}}-\widetilde{W}_{t_i}$. By \cite{[BE08],[A13]} this approximation is convergent:
	\begin{theorem}
		Assume $\epsilon,\Delta t>0$, $\zeta\in\{0,1\}$, and (\textbf{H1}). Then there is $C>0$, independent \footnote{The constant $C$ is independent also of $\epsilon$ by Remark 4.1 in \cite{[A13]}.} of $\Delta t,\epsilon,\zeta$, such that
for $ i=0,\dots,N-1$, 		\begin{equation}\label{estiX}
	\mathbb{E}\bigg[\sup_{t\in[t_i,t_{i+1}]}\big|X^\epsilon_t-X_{t_i}^{\epsilon,\Delta t}\big|^2\bigg]
 \leq C
		(1+|x|^2)
		\Delta t.\end{equation}
	\end{theorem}
Next we want to approximate the backward SDE \eqref{FBSDE2}. 
	As opposed to the local case we now need also approximations of $U^\epsilon$ which has an extra variable $e$ compared to $Y^\epsilon,Z^\epsilon,L^\epsilon$. To determine $e$-dependent approximations of $U^\epsilon$ we subdivide the domain $E_\epsilon$ and discretise the stochastic and deterministic $e$-integrals. 
	For $h>0$, let $\mathcal{G}_h:=\{K_j: j=1, 2,\dots\}$ be a partition of $E^\epsilon$ satisfying the following disjoint covering, refinement, and convergence conditions:
	\begin{enumerate}[(D1)]
		\item $K_{i}\cap K_{j}\underset{i\ne j}{=}\emptyset\,,\, \displaystyle\bigcup_{j\in\mathbb{N}}K_j=E^\epsilon$, and for all $R>0$,\, $\displaystyle\sup_{j}\text{diam}(B_R\cap K_j)\to 0 \;\text{as}\; h\to 0^+$.
		\item For $\phi\in \mathbb{L}^2_{\nu,[0,T],E}$ and $j\in\mathbb N$ with $\nu(K_j)\ne 0$, define $$\phi^{j}:=\displaystyle\frac{1}{\nu(K_j)}\mathbb{E}\bigg[\int_{K_j}\phi_s(e)\nu(de)\bigg]\qquad\text{and}\qquad \displaystyle\gamma_j:=\frac{1}{\nu(K_j)}\int_{K_j}\gamma(e)\nu(de).$$ Then we have,
		\begin{eqnarray}
		&&\mathbb{E}\bigg[\displaystyle\sum_{j\geq 1}\int_{K_j}|\phi(e)-\phi^{j}|^2\nu(de)\bigg]\to 0\quad\text{as}\quad h\to0^+,\nonumber
		\\&& \displaystyle\sum_{j\geq 1}\int_{K_j}|\gamma(e)-\gamma_j|^2\nu(de)\to 0\quad\text{as}\quad h\to0^+.\nonumber
		\end{eqnarray}
	\end{enumerate}
For the backward SDE \eqref{FBSDE2}, we then consider the following Euler in time approximation where the two $e$-integrals have been approximated by quadratures, 
\begin{equation}\label{DisBSDE}
\begin{array}{ll}
Y^\epsilon_{t_i}\sim  g(X_{t_N}^{\epsilon,\Delta t})+\displaystyle\sum_{l=i}^{N-1}f\bigg(t_l,X^{\epsilon,\Delta t}_{t_l},Y^\epsilon_{t_l},Z^\epsilon_{t_l},\sum_{j\geq 1}U^\epsilon_{t_l}(e_j)\gamma_j\nu(K_j)+\zeta D_e\gamma(0)\Sigma^\frac{1}{2}_\epsilon L^\epsilon_{t_l} \bigg)\Delta t_{l}\\\hspace{3.5cm} -\displaystyle\sum_{l=i}^{N-1}Z^\epsilon_{t_l}\Delta W_{t_l}-\zeta\sum_{l=i}^{N-1}L^\epsilon_{t_l}\Delta \widetilde{W}_{t_l}-\sum_{l=i}^{N-1}\sum_{j\geq 1}U^\epsilon_{t_l}(e_j)\tilde{\mu}(K_j,(t_l,t_{l+1}]),
\end{array}
\end{equation}
where $e_j\in K_j.$ Note that no processes that satisfies (\ref{DisBSDE}) as an equality can be adapted. This problem can be fixed by taking conditional expectation as in \cite{[BE08]}, but this is too costly in high dimensions. This problem can be overcome here using deep-neural net approximations.

\begin{remark}
The quadrature methods above are stochastic integrals of elementary/piecewise constant in $(t,e)$ integrands, and all increments are independent: $\Delta W_{t_n},\Delta \widetilde W_{t_k}, \tilde{\mu}(K_j,(t_l,t_{l+1}])$ for all $n,k,j,l$. Independence of increments follows by definition of the Levy process, e.g. the compensated random measure $\tilde \mu$ is independently scattered -- evaluations on disjoint set $K_j$'s give independent random variables. The total approximation is similar to the approximate integrals used in the construction of the It\^{o} integral with respect to a Levy proecess. We refer to \cite{ApBook09} for more details.
\end{remark}	
	
	\subsection{Neural network architecture}\label{sec:DNN}
	Let $x=\left(x_{1}, \ldots, x_{q}\right)^{\top} \in \mathbb{R}^{q} .$ The neural network is defined as follow
	$$
	\mathbb{N N}(x, \theta)=A_{L} \circ \rho \circ A_{L-1} \circ \rho \ldots \circ A_{1}\left(\left(x_{1}, \ldots, x_{q}\right)^{\top}\right)
	$$
	where $A_{l}(y)=P_{l} y+\beta_{l}$ for $l=1, \ldots, L$, $P_{1} \in \mathbb{R}^{q \times m}$, $P_{l} \in \mathbb{R}^{m \times m}$ for $l=2, \ldots, L-1$,
	$P_{L} \in \mathbb{R}^{m \times 1}$, $\beta_{l} \in \mathbb{R}^{m}$ for $l=1, \ldots, L-1$, and $\beta_{L} \in \mathbb{R}$. $L$ is the number of layers and $m$ is the number of neurons per layer (that we assume to be the same for every layer). The $P_{l}$ correspond to the weights and $\beta_{l}$ to the bias. The activation function $\rho$ is  chosen as the ReLu function,  $\rho(x)=\max (0, x).$ The collection of weights and biases are the parameters of the neural network, $\theta:=\left(P_{1}, \ldots, P_{L}, \beta_{1}, \ldots, \beta_{L}\right).$
	
Deep neural networks are very efficient for approximations of functions even in high-dimensional
spaces. The following universal approximation theorem is from \cite{[HSW90]}:	
\begin{theorem}\label{UAthm}
  The set of all
	neural network functions is dense in $L^2(\mu)$ for any finite measure $\mu$ on $\mathbb{R}^d
	, d > 0$ whenever
	the activation function  $\rho$ is continuous and non-constant.
\end{theorem}

	\subsection{Deep learning approximations.}
	 We now solve \eqref{DisBSDE} in a relaxed way using an adapted neural network regression. In a backward in time recursion, we find the best (adapted) approximations by neural network functions $\mathcal{Y}_{i}$, $\mathcal{Z}_{i}$, $\mathcal{W}_{i}$ and $\mathcal{U}_i$, $i =0,\dots,N-1,$ that minimize an $L^2$ residual error related to \eqref{DisBSDE}. 
  
  At time $t=t_i$, let $(\theta_i,\eta_i,\lambda_i,\alpha_i)$ denote the 
  neural network parameters,  $(\theta^*_i,\eta^*_i,\lambda_i^*,\alpha^*_i)$ optimal parameters, and optimal neural networks,
	\begin{eqnarray}
	&\hat{\mathcal{Y}}_{i}(x):=\mathcal{Y}_{i} (x,\theta_i^*)&\simeq u^\epsilon (t_i,x),
	\nonumber\\
	&\hat{\mathcal{Z}}_{i}(x):=\mathcal{Z}_{i}(x,\eta_i^*)&\simeq \sigma^T(x)D_x u^\epsilon(t_i,x),
	\nonumber\\
	&\hat{\mathcal{W}}_{i}(x):=\mathcal{W}_{i}(x,\lambda_i^*)&\simeq{(\Sigma^{\frac{1}{2}}_\epsilon})^T D^T_e\beta(x,0)D_x u^\epsilon (t_i,x),
	\nonumber\\
	&\hat{\mathcal{U}}_{i}(x,e):=\mathcal{U}_{i}(x,e,\alpha_i^*) &\simeq u^\epsilon(t_i,x+\beta(x,e))-u^\epsilon(t_i,x),\nonumber
	\end{eqnarray}
	where the approximate equal signs holds by Lemma \ref{u=Y} when $u^\epsilon$ is smooth enough.
 Given $(\theta^*_j,\eta^*_j,\lambda_j^*,\alpha^*_j)$ (and $(\hat{\mathcal{Y}}_{j}$, $\hat{\mathcal{Z}}_{j}$, $\hat{\mathcal{W}}_{j}$, $\hat{\mathcal{U}}_j)$) for $j=i+1,\dots N$,  we determine $(\theta^*_i,\eta^*_i,\lambda_i^*,\alpha^*_i)$ (and hence also $\hat{\mathcal{Y}}_{i}$, $\hat{\mathcal{Z}}_{i}$, $\hat{\mathcal{W}}_{i}$, $\hat{\mathcal{U}}_i$) by minimizing  an  $L^2$ residual error function  related to  (\ref{DisBSDE}):
	\begin{align}\label{Minfun}
&\mathcal{R}_i(\theta,\eta,\lambda,\alpha):=\\[-0.1cm]
&\ \mathbb{E}\,\Big|\,\mathcal{Y}_i(X^{\epsilon,\Delta t}_{t_i},\theta)-g(X^{\epsilon,\Delta t}_{t_N}) -F(t_i,X^{\epsilon,\Delta t}_{t_i},\theta,\eta,\lambda,\alpha) - \sum_{l=i+1}^{N-1}F(t_l,X^{\epsilon,\Delta t}_{t_l},\theta_l^*,\eta_l^*,\lambda_l^*,\alpha_l^*)\Big|^2,\nonumber
	\end{align}
	where 
	\begin{align*}
 &F(t_i,X^{\epsilon,\Delta t}_{t_i},\theta,\eta,\lambda,\alpha)\\ 
 &\quad:=
f\Big(t_i,X^{\epsilon,\Delta t}_{t_i},\mathcal{Y}_i(X^{\epsilon,\Delta t}_{t_i},\theta),\mathcal{Z}_i(X^{\epsilon,\Delta t}_{t_i},\eta),\\
&\qquad\qquad\qquad\qquad\qquad\sum_{j\geq 1}\mathcal{U}_{i}(X^{\epsilon,\Delta t}_{t_i},e_j,\alpha)\gamma_j\nu(K_j)+\zeta D_e\gamma(0)\Sigma^\frac{1}{2}_\epsilon \mathcal{W}_{i}(X^{\epsilon,\Delta t}_{t_i},\lambda)\Big)\Delta t_{i}\\ 
 &\qquad\quad - \ \zeta\mathcal{W}_{i}(X^{\epsilon,\Delta t}_{t_i},\lambda)\Delta \widetilde{W}_{t_i} \ - \ \mathcal{Z}_{i}(X^{\epsilon,\Delta t}_{t_i},\eta)\Delta W_{t_i} \ - \ \sum_{j\geq 1}\mathcal{U}_{i}(X^{\epsilon,\Delta t}_{t_i},e_j,\alpha)\tilde{\mu}(K_j,(t_i,t_{i+1}]).
	\end{align*}\ee

	\begin{algorithm}
		\caption{\ A DeepBSDE method for \eqref{PIDE}.}
		\begin{algorithmic}
			\label{ALGO1}
		\smallskip	\renewcommand{\algorithmicrequire}{\textbf{Input:}}
	\renewcommand{\algorithmicensure}{\textbf{Output:}}
		\REQUIRE
		\STATE \textit{Initialize} $\hat{\mathcal{Y}}_N(.)=g(.)$
		\\
		\FOR {$i = N-1,\dots, 0$}
		\STATE Compute:
		\\[0.2cm] \quad $(\theta^*_i,\eta^*_i,\lambda^*_i,\alpha^*_i)\in \text{argmin}_{(\theta,\eta,\lambda,\alpha)}\mathcal{R}_i(\theta,\eta,\lambda,\alpha).$
		\medskip
	\STATE Update:\\ \medskip \quad $(\hat{\mathcal{Y}},\hat{\mathcal{Z}},\hat{\mathcal{W}},\hat{\mathcal{U}})_i=(\,\mathcal{Y}_i(\cdot,\theta_i^*),\mathcal{Z}_i(\cdot,\eta_i^*), \mathcal{W}_i(\cdot,\lambda_i^*),\mathcal{U}_i(\cdot,\alpha_i^*)\,).$
	\medskip
	\ENDFOR
	\smallskip
	\ENSURE
	\,\\ \medskip
	$(\hat{\mathcal{Y}},\hat{\mathcal{Z}},\hat{\mathcal{W}},\hat{\mathcal{U}})_{i=0}^N\quad\text{approximation of}\quad (Y^\epsilon_{t_i},Z^\epsilon_{t_i},L^\epsilon_{t_i},U^\epsilon_{t_i}\,)_{i=0}^N.$
	\end{algorithmic}
	\end{algorithm}
	\begin{remark}
 (a)  The approximation $(\,\hat{\mathcal{Y}}_i, \hat{\mathcal{Z}}_i, \hat{\mathcal{W}}_i, \hat{\mathcal{U}}_i\,)(X^{\epsilon,\Delta t}_{t_i})$ of (\ref{DisBSDE}) is $\mathcal F_{t_i}$-measurable and hence adapted. \ee Our procedure therefore resembles taking conditional expectations of (\ref{DisBSDE}), but the (best) approximation is in a much smaller class of functions.
\smallskip

\noindent (b) When $\theta,\eta,\lambda,\alpha=0$, $\mathcal{Y},\mathcal{Z},\mathcal{W},\mathcal{U}=0$ and $\mathcal{R}_i$ is finite, and hence the infimum of $\mathcal{R}_i$ exists. Without loss of generality, we assume that the infimum is achieved and is a minimum. Otherwise, we simply work with $\epsilon$-optimal parameters for $\mathcal{R}_i$, i.e. we use the fact that for any $\epsilon>0$, there is   $(\theta^\epsilon,\eta^\epsilon,\lambda^\epsilon,\alpha^\epsilon)$ satisfying
			$$\mathcal{R}_i(\theta^\epsilon,\eta^\epsilon,\lambda^\epsilon,\alpha^\epsilon)\leq \mathcal{R}_i(\theta,\eta,\lambda,\alpha)+\epsilon. \qquad \text{for all}\; \theta,\eta,\lambda,\alpha.$$
(c) The optimal parameters can be approximated through  stochastic gradient descent-type (SGD) algorithms. See e.g. \cite{BF11,FGJ20} for more details.
	\end{remark}
	\subsection{Intermediate approximation and errors.}\label{sec:interm}
Following \cite{HPW20} (see also \cite{[HL18]}), we will show convergence of Algorithm \ref{ALGO1} adapting some of the machinery of \cite{[BT04]}. To do that we derive an approximation which is intermediate between the FBSDE \eqref{FSDE2}--\eqref{FBSDE2} and our Deep BSDE approximation in Algorithm \ref{ALGO1}. 
This scheme shares features with the approximations in \cite{[BE08],[BT04]}, and is a direct extension of the intermediate scheme of \cite{HPW20}. 
Note the modifications due to the compensated Poisson random measure $\tilde\mu$. Compared to \cite{[BE08]}, we consider unbounded intensity/Levy measures (infinite activity of jumps). As opposed to \cite{[BE08]} where $U$ only appears through an integral, here we need a subdivision of $e$-space and pointwise a.e. definition of $U$. \ee We denote by  $\mathbb{E}_i$ the conditional expectation given $\mathcal{F}_{t_i}$.\\[0.2cm]
1. Taking conditional expectations in (\ref{DisBSDE}) we obtain
	\begin{align*}
	Y^\epsilon_{t_i}\sim &\ \mathbb{E}_i\bigg[g(X_{t_N}^{\epsilon,\Delta t})+\displaystyle\sum_{l=i+1}^{N-1}f\bigg(t_l,X^{\epsilon,\Delta t}_{t_l},Y^\epsilon_{t_l},Z^\epsilon_{t_l},\sum_{j\geq 1}U^\epsilon_{t_l}(e_j)\gamma_j\nu(K_j)
	+\zeta D_e\gamma(0)\Sigma^\frac{1}{2}_\epsilon L^\epsilon_{t_l}\bigg)\Delta t_{l}\bigg]
	\\ 
	&\ +f\bigg(t_i,X^{\epsilon,\Delta t}_{t_i},Y^\epsilon_{t_i},Z^\epsilon_{t_i},\displaystyle\sum_{j\geq 1}U^\epsilon_{t_i}(e_j)\gamma_j\nu(K_j)+\zeta D_e\gamma(0)\Sigma^\frac{1}{2}_\epsilon L^\epsilon_{t_i}\bigg)\Delta t_{i},
\end{align*}
	since $\sum Z^\epsilon\Delta W$, $\sum L^\epsilon \Delta\widetilde{W}$, and $\sum\sum U^\epsilon \tilde{\mu}(\dots)$ are mean zero martingales independent of $\mathcal{F}_{t_i}$.\\[0.2cm]
2. First pre-multiplying (\ref{DisBSDE}) by $\Delta W_{t_i}$, $\Delta \widetilde W_{t_i}$, or  $\tilde{\mu}(K_j,[t_i,t_{i+1}))$ for  $j=1,\dots,$ and then taking conditional expectations gives
	\begin{align*}
	0\sim &\ \mathbb{E}_i\bigg[g(X_{t_N}^{\epsilon,\Delta t})\Delta W_{t_i}+\displaystyle\sum_{l=i+1}^{N-1}f\Big(t_l,X^{\epsilon,\Delta t}_{t_l},Y^\epsilon_{t_l},Z^\epsilon_{t_l},\sum_{j\geq 1}U^\epsilon_{t_l}(e_j)\gamma_j\nu(K_j)\\ \nonumber
	&\qquad +\zeta D_e\gamma(0)\Sigma^\frac{1}{2}_\epsilon L^\epsilon_{t_l}\Big)\Delta t_{l}\Delta W_{t_i}\bigg]-Z^\epsilon_{t_i}\Delta t_{i},\end{align*}
 \begin{align*}
	0\sim & \  \mathbb{E}_i\bigg[g(X_{t_N}^{\epsilon,\Delta t})\Delta \widetilde W_{t_i}+\displaystyle\sum_{l=i+1}^{N-1}f\Big(t_l,X^{\epsilon,\Delta t}_{t_l},Y^\epsilon_{t_l},Z^\epsilon_{t_l},\sum_{j\geq 1}U^\epsilon_{t_l}(e_j)\gamma_j\nu(K_j)\\ \nonumber
	& \qquad+\zeta D_e\gamma(0)\Sigma^\frac{1}{2}_\epsilon L^\epsilon_{t_l}\Big)\Delta t_{l}\Delta \widetilde W_{t_i}\bigg]-\zeta L^\epsilon_{t_i}\Delta t_{i},\end{align*}
 \begin{align*}
	0\sim &\ \mathbb{E}_i\bigg[g(X_{t_N}^{\epsilon,\Delta t})\tilde{\mu}(K_j,[t_i,t_{i+1}))+\displaystyle\sum_{l=i+1}^{N-1}f\Big(t_l,X^{\epsilon,\Delta t}_{t_l},Y^\epsilon_{t_l},Z^\epsilon_{t_l},\displaystyle\sum_{j\geq 1}U^\epsilon_{t_l}(e_j)\gamma_j\nu(K_j)\\ \nonumber
	&\qquad+\zeta D_e\gamma(0)\Sigma^\frac{1}{2}_\epsilon L^\epsilon_{t_l}\Big)\Delta t_{l}\tilde{\mu}(K_j,[t_i,t_{i+1}))\bigg]-\Delta t_i\nu(K_j)U^\epsilon_{t_i}(e_j),
\end{align*}
	where we used the It\^{o} isometry and the pairwise independence between $\Delta W_i$, $\Delta W_j$, $\Delta \widetilde W_i$, $\Delta \widetilde W_j$, $\tilde{\mu}(K_k,[t_i,t_{i+1}))$, and $\tilde{\mu}(K_l,[t_j,t_{j+1}))$ for all $l\neq k$ and $i\neq j$.\\ 
	
\noindent 3. We introduce an auxiliary process $\mathcal{V}_{t_i}$ (cf. \cite{[HL18],HPW20}
) to approximate (integrands of) the conditional expectations in step 1 and 2:
	\begin{align}\label{Vi}
\mathcal{V}_{t_i} & :=g(X_{t_N}^{\epsilon,\Delta t})+\displaystyle\sum_{l=i}^{N-1}f\Big(t_l,X^{\epsilon,\Delta t}_{t_l},\hat{\mathcal{Y}}_{t_{l}}(X^{\epsilon,\Delta t}_{t_l}),\hat{\mathcal{Z}}_{t_{l}}(X^{\epsilon,\Delta t}_{t_l}),\\ \nonumber
&\hspace{4cm}\displaystyle\sum_{j\geq 1}\hat{\mathcal{U}}_{t_{l}}(X^{\epsilon,\Delta t}_{t_l},e_j)\gamma_j\nu(K_j)+\zeta D_e\gamma(0)\Sigma^\frac{1}{2}_\epsilon \hat{\mathcal{W}}_{l}(X^{\epsilon,\Delta t}_{t_l})\Big)\Delta t_{l},
\end{align}
for $i=1,\dots,N$. 
Our intermediate scheme is then defined  as the following backward Euler projection scheme suggested by step 1 and 2:
	\begin{equation}\label{NewBack}
	\begin{cases} \hat{\mathcal{V}}_{t_i}=\mathbb{E}_i\big[\mathcal{V}_{t_{i+1}}\big] 
	+f\Big(t_i,X^{\epsilon,\Delta t}_{t_i},\hat{\mathcal{V}}_{t_i},\bar{\hat{Z}}_{t_i},\displaystyle\sum_{j\geq 1}\bar{\hat{U}}_{t_i}(e_j)\gamma_j\nu(K_j)+\zeta  D_e\gamma(0)\Sigma^\frac{1}{2}_\epsilon\bar{\hat{L}}_{t_i}\Big)\Delta t_{i},\\[0.8cm]
	\bar{\hat{Z}}_{t_i}=\displaystyle\frac{1}{\Delta t_i}\mathbb{E}_i\big[\mathcal{V}_{t_{i+1}}\Delta W_{t_i}\big],\\[0.5cm]
	\bar{\hat{L}}_{t_i}=\displaystyle\frac{1}{\Delta t_i}\mathbb{E}_i\big[\mathcal{V}_{t_{i+1}}\Delta \widetilde W_{t_i}\big],\\[0.5cm]
	\bar{\hat{U}}_{t_i}(e_j)=\displaystyle\frac{1}{\Delta t_i\nu(K_j)}\mathbb{E}_i\Big[\mathcal{V}_{t_{i+1}}\tilde{\mu}(K_j,[t_i,t_{i+1}))\Big],
	\end{cases}
	\end{equation}
	for $i=0,\dots,N-1,$ and $j=1,\dots$. This scheme is intermediate between Algorithm \ref{ALGO1} and the time discrete "FBSDEs" \eqref{DisSDE}--\eqref{DisBSDE}.

Next we define
\begin{align}
    \label{shema2.1} 
\tilde{\mathcal{V}}_{t_i} &:=\mathbb{E}_i\big[\mathcal{V}_{t_{i}}\big] =\mathbb{E}_i\bigg[\tilde{\mathcal{V}}_{t_{i+1}}+f\Big( t_i,X^{\epsilon,\Delta t}_{t_i},\hat{\mathcal{Y}}_{t_{i}}(X^{\epsilon,\Delta t}_{t_i}),\hat{\mathcal{Z}}_{t_{i}}(X^{\epsilon,\Delta t}_{t_i}),\\ \nonumber &\hspace{3cm}\displaystyle\sum_{j\geq 1}\hat{\mathcal{U}}_{t_{i}}(X^{\epsilon,\Delta t}_{t_i},e_j)\gamma_j\nu(K_j) +\zeta D_e\gamma(0)\Sigma^\frac{1}{2}_\epsilon \hat{\mathcal{W}}_{i}(X^{\epsilon,\Delta t}_{t_i})\Big)\Delta t_{i}\bigg].
\end{align}
 The last equality follows by the tower property of condition expectation and the definition of $\tilde{\mathcal{V}}_{t_{i+1}}$. 
By properties of conditional expectations, we can replace $\mathcal{V}_{t_{i+1}}$ by  $\tilde{\mathcal{V}}_{t_{i+1}}$ in \eqref{NewBack}, e.g. 
	\begin{align}\label{tilda projection}
	&\bar{\hat{Z}}_{t_i}=\frac{\mathbb{E}_i\big[\tilde{\mathcal{V}}_{t_{i+1}}\Delta W_{t_i}\big]}{\Delta t_i},\quad\bar{\hat{L}}_{t_i}=\frac{\mathbb{E}_i\big[\tilde{\mathcal{V}}_{t_{i+1}}\Delta \widetilde W_{t_i}\big]}{\Delta t_i},
	\quad \bar{\hat{U}}_{t_i}(e_j)=\frac{\mathbb{E}_i\big[\tilde{\mathcal{V}}_{t_{i+1}}\tilde{\mu}(K_j,[t_i,t_{i+1}))\big]}{\Delta t_i\nu(K_j)}.
	\end{align}	

\noindent 4. Since $X^{\epsilon,\Delta t}$ is a Markov process,  $\mathbb{E}_i[X^{\epsilon,\Delta t}_{t_j}]= \mathbb{E}[X^{\epsilon,\Delta t}_{t_j}|\sigma(X^{\epsilon,\Delta t}_{t_i})]$, the Doob-Dynkin Lemma (\cite{[K95]}, p.$90$) gives measurable deterministic functions $\hat{v}_i(x), \bar{\hat{z}}_i(x)$, $\bar{\hat{l}}_i(x)$, $\bar{\hat{\rho}}_i(x, e_j)$ such that
	\begin{equation}\label{DDinterm}
	\hat{\mathcal{V}}_{t_i}=\hat{v}_i(X^{\epsilon,\Delta t}_{t_i}),\quad \bar{\hat{Z}}_{t_i}=\bar{\hat{z}}_i(X^{\epsilon,\Delta t}_{t_i}), \quad \bar{\hat{L}}_{t_i}=\bar{\hat{l}}_i(X^{\epsilon,\Delta t}_{t_i}),\quad\text{and}\quad \bar{\hat{U}}_{t_i}(e_j)=\bar{\hat{\rho}}_i(X^{\epsilon,\Delta t}_{t_i}, e_j).
	\end{equation}
	So the random variables $\hat{\mathcal{V}}_{t_i}$, $\bar{\hat{Z}}_{t_i}$, $\bar{\hat{L}}_{t_i}$ and $\bar{\hat{U}}_{t_i}(e_j)$ are functions of $X^{\epsilon,\Delta t}_{t_i}$ for each $i = 0,\dots,N-1.$ \\

\noindent 5.  By 
\eqref{Vi}  and the It\^{o}-Levy representation theorem, see e.g.~Theorem 5.3.5 in \cite{ApBook09}, there exists processes $\hat{Z}\in\mathbb{H}^2$, $\hat{L}\in\mathbb{H}^2$,  and $\hat{U}\in\mathbb{L}_\nu^2$ 
	such that 
    $ {\mathcal{V}}_{t_{i+1}} = \mathbb{E}_i[{\mathcal{V}}_{t_{i+1}}]+\int_{t_i}^{t_N} \hat{Z} dW+\zeta \int_{t_i}^{t_N}\hat{L} d\widetilde W +\int_{t_i}^{t_N}\sum \int_{K_j} \hat U \tilde \mu$.  Then using the expressions for ${\mathcal{V}}_{t_{i+1}}$ in \eqref{Vi} and $\mathbb{E}_i[{\mathcal{V}}_{t_{i+1}}]$ given by the first line of
    \eqref{NewBack}, we see that
\begin{align}
\begin{split}\label{eq-representation}
 {\mathcal{V}}_{t_{i+1}}&= g(X_{t_N}^{\epsilon,\Delta t})+\displaystyle\sum_{l=i+1}^{N-1}f\Big(t_l,X^{\epsilon,\Delta t}_{t_l},\hat{\mathcal{Y}}_{l}(X^{\epsilon,\Delta t}_{t_l}),\hat{\mathcal{Z}}_{l}(X^{\epsilon,\Delta t}_{t_l}),\\
&\hspace{3.8cm}\displaystyle\sum_{j\geq 1}\hat{\mathcal{U}}_{l}(X^{\epsilon,\Delta t}_{t_l},e_j)\gamma_j\nu(K_j)+\zeta  D_e\gamma(0)\Sigma^\frac{1}{2}_\epsilon \hat{\mathcal{W}}_{l}(X^{\epsilon,\Delta t}_{t_l})\Big)\Delta t_{l}\\ &=\hat{\mathcal{V}}_{t_i}-f\Big(t_i,X^{\epsilon,\Delta t}_{t_i},\hat{\mathcal{V}}_{t_i},\bar{\hat{Z}}_{t_i},\displaystyle\sum_{j\geq 1}\bar{\hat{U}}_{t_i}(e_j)\gamma_j\nu(K_j)+\zeta D_e\gamma(0)\Sigma^\frac{1}{2}_\epsilon\bar{\hat{L}}_{t_i}\Big)\Delta t_i\\
&\quad+\integ{t_i}{t_{N}}\hat{Z}_sdW_s+\zeta\integ{t_i}{t_{N}}\hat{L}_sd \widetilde W_s+\sum_{j\geq 1}\integ{t_i}{t_{N}}\int_{K_j}\hat{U}_s(e)\tilde{\mu}(de,ds).
	\end{split}
\end{align}
	From (\ref{eq-representation}) and by the It\^{o} isometry, we find that $\bar{\hat{Z}}$, $\bar{\hat{L}}$ and $\bar{\hat{U}}$ defined in (\ref{NewBack}) satisfy
	\begin{equation}\label{piZLU}
	\bar{\hat{Z}}_{t_i}=\pi_ {[t_i,t_{i+1}]}(\hat{Z}),\qquad \bar{\hat{L}}_{t_i}=\pi_ {[t_i,t_{i+1}]}(\hat{L}),\qquad\text{and}\qquad\bar{\hat{U}}_{t_i}(e_j)=\pi_ {[t_i,t_{i+1}],K_j}(\hat{U}),
	\end{equation}
	for $i=0,\dots,N-1$ and $j=1,\dots$, where the $\pi_\cdot$-projections are given in Definition \ref{def:proj}.\\

	\noindent 6. For the processes $(Z^\epsilon, L^\epsilon, U^\epsilon)$ we define the correspond (mean-square) projection errors by
	\begin{equation}\label{def:R}
	\begin{cases}
	\mathcal{R}^2_{Z^\epsilon}(\Delta t):=\mathbb{E}\bigg[\displaystyle\sum_{i=0}^{N-1}\integ{t_i}{t_{i+1}}\big|Z^\epsilon_t-\pi_ {[t_i,t_{i+1}]}(Z^\epsilon)\big|^2dt\bigg],\\[0.4cm] \mathcal{R}^2_{L^\epsilon}(\Delta t):=\mathbb{E}\bigg[\displaystyle\sum_{i=0}^{N-1}\integ{t_i}{t_{i+1}}\big|L^\epsilon_t-\pi_ {[t_i,t_{i+1}]}(L^\epsilon)\big|^2dt\bigg],
\\[0.4cm]
	\mathcal{R}^2_{U^\epsilon}(\Delta t,h):=\mathbb{E}\bigg[\displaystyle\sum_{i=0}^{N-1}\displaystyle\sum_{j\geq 1}\integ{t_i}{t_{i+1}}\int_{K_j}\big|U^\epsilon_t(e)-\pi_ {[t_i,t_{i+1}],K_j}(U^\epsilon)\big|^2\nu(de)dt\bigg].
\end{cases}
	\end{equation}
When $(Z^\epsilon, L^\epsilon, U^\epsilon)$ belong to appropriate $L^2$-spaces, these errors converge to zero as $\Delta t,h\to 0$. This is the case when Theorem \ref{WP_FBSDE} holds.	
	\subsection{Convergence of the DeepBSDE scheme}
The main results of this paper are the following two convergence results for Algorithm \ref{ALGO1}. The first result gives a strong/$L^2$ error bound for the convergence of the DeepBSDE solutions to the solutions of the approximate FBSDE \eqref{FSDE2}--\eqref{FBSDE2}.
	\begin{theorem}\label{theo-convergence1}
		Assume $\epsilon,\Delta t,h>0$,  $\zeta\in\{0,1\}$, (\textbf{H1}), and (\textbf{H2}). If $(X^{\epsilon,\Delta t},\hat{\mathcal{Y}},\hat{\mathcal{Z}},\hat{\mathcal{W}},\hat{\mathcal{U}})$ and $(X^\epsilon,Y^\epsilon,Z^\epsilon,L^{\epsilon},U^\epsilon)$ solve Algorithm \ref{ALGO1} and FBSDE \eqref{FSDE2}--\eqref{FBSDE2} respectively, then there is $C>0$ independent of  $\epsilon,\zeta,\Delta t,h$ such that
		\begin{align} \nonumber
		&\max_{i=0,\dots,N-1}\mathbb{E}\big|Y^\epsilon_{t_i}-\hat{\mathcal{Y}}_{i}(X^{\epsilon,\Delta t}_{t_i})\big|^2 \ + \ \mathbb{E}\bigg[\displaystyle\sum_{i=0}^{N-1}\integ{t_i}{t_{i+1}}\big|Z^\epsilon_t-\hat{\mathcal{Z}_i}(X^{\epsilon,\Delta t}_{t_i})\big|^2dt\bigg]\\
		&\quad+\ \mathbb{E}\bigg[\displaystyle\sum_{i=0}^{N-1}\displaystyle\integ{t_i}{t_{i+1}}\Big|\sum_{j\geq 1}\int_{K_j}\big|U^\epsilon_t(e)-\hat{\mathcal{U}}_i(X^{\epsilon,\Delta t}_{t_i},e_j)\big|\gamma_j\nu(de)\Big|^2dt\bigg]\label{Estimate Y}\\[0.2cm] 
		\nonumber
		&\leq C\Big((1+|x|^2)\Delta t+\mathcal{R}^2_{Z^\epsilon}(\Delta t)+\zeta\mathcal{R}^2_{L^\epsilon}(\Delta t)+\mathcal{R}^2_{U^\epsilon}(\Delta t,h)+\mathcal{R}^2_\gamma(h) \\ \nonumber
		&\hspace{3.5cm}+ \displaystyle\sum_{i=0}^{N-1}(\epsilon_i^{\mathcal{N},v}+\Delta t\epsilon_i^{\mathcal{N},z}+\zeta\Delta t\epsilon_i^{\mathcal{N},l}+\Delta t\epsilon_i^{\mathcal{N},\rho})\Big).
		\end{align}
		where $\mathcal{R}^2_{Z^\epsilon}, \mathcal{R}^2_{L^\epsilon},\mathcal{R}^2_{U^\epsilon}$ are defined in \eqref{def:R}, $\mathcal{R}^2_\gamma(h):=\displaystyle\sum_{j\geq 1}\int_{K_j}\big|\gamma(e)-\gamma_j\big|^2\nu(de)$, 
		\begin{align*} &\epsilon_i^{\mathcal{N},v}:=\displaystyle\inf_{\theta}\mathbb{E}\big|\hat{v}_i(X^{\epsilon,\Delta t}_{t_i})-{\mathcal{Y}}_{i}(X^{\epsilon,\Delta t}_{t_i},\theta)\big|^2,\quad \epsilon_i^{\mathcal{N},z}:=\displaystyle\inf_{\eta}\mathbb{E}\big|\bar{\hat{z}}_i(X^{\epsilon,\Delta t}_{t_i})-{\mathcal{Z}}_{i}(X^{\epsilon,\Delta t}_{t_i},\eta)\big|^2,\\[0.2cm] &\epsilon_i^{\mathcal{N},l}:=\displaystyle\inf_{\lambda}\mathbb{E}\big|\bar{\hat{l}}_i(X^{\epsilon,\Delta t}_{t_i})-{\mathcal{W}}_{i}(X^{\epsilon,\Delta t}_{t_i},\lambda)\big|^2,\quad \epsilon_i^{\mathcal{N},\rho}:=\displaystyle\inf_{\alpha}\displaystyle\mathbb{E}\Big[\sum_{j\geq 1}\big|\bar{\hat{\rho}}_i(X^{\epsilon,\Delta t}_{t_i},e_j)-{\mathcal{U}}_{i}(X^{\epsilon,\Delta t}_{t_i},e_j,\alpha)\big|^2\nu(K_j)\Big],
		\end{align*}
		and $\hat{v}_i,\dots,\bar{\hat{\rho}}_i$ are Doob-Dynkin representations defined in \eqref{DDinterm} of the solutions of the intermediate scheme \eqref{NewBack}.
	\end{theorem}

\begin{remark}
(a)\ The first 4 error terms represent discretization errors in time (and jumps). They are discretisation errors when the intermediate scheme \eqref{NewBack} is seen as an approximation of the FBSDE \eqref{FSDE2}--\eqref{FBSDE2}. The $\mathcal{R}$-terms are projection errors \eqref{def:R}, and converge to zero as $\Delta t,h\to 0$ by Theorem \ref{WP_FBSDE}. 

As explained in Section \ref{sec:interm}, these terms come from an adaptation of the analysis of \cite{[BE08]}. Differently from \cite{[BE08]}, we have approximated the infinite activity part of the Poisson random integral by a Brownian motion, discretized the finite activity  part in jump space, and give a pointwise approximation of the $U$-component of the FBSDE.
\smallskip

\noindent (b)  $\mathcal{R}^2_\gamma(h)$ is  the (quadrature) approximation error of the $L^2$-integral of $\gamma$, which converges to
		zero as $h$ goes to zero by assumptions \eqref{levy} in (H1) and (H2) (ii). This error is due to our approximation of jump space.
		
\smallskip\noindent (c) The last 4 error terms are $L^2$-projection errors, or best approximation errors, when solutions of the intermediate scheme \eqref{NewBack} are approximated by neural nets. 
By the universal approximation property (Theorem \ref{UAthm}), they converge to zero as the number of neurons $m\to \infty$ (Section \ref{sec:DNN}). 
\end{remark}

	\begin{remark}
As opposed to previous works, we  approximate $U_\epsilon$ and prove error bounds in $L^2((0,T) \times \Omega; L^1(E; \gamma \nu))$. This is consistent with (but stronger than) \cite{[BE08]} where integrals $\Gamma_\epsilon=\int_E U_\epsilon\gamma\nu$ are proven to converge in $L^2((0,T) \times \Omega)$.
Under additional assumptions on the subdivisions of $E$, we prove a stronger convergence result for the more natural space of $L^2(\Omega\times(0,T)\times E;P\otimes dt\otimes \nu)$ in Appendix \ref{appB}. The point is that we must refine more and more near the singularity $e=0$. To be precise,
if there is $k_h>0$ such that $\nu(K_j)\gamma^2_j\geq k_h$, then
	   \begin{align} &k_h\;\mathbb{E}\bigg[\displaystyle\sum_{i=0}^{N-1}\displaystyle\sum_{j\geq 1}\integ{t_i}{t_{i+1}}\int_{K_j}\big|U^\epsilon_t(e)-\hat{\mathcal{U}}_i(X^{\epsilon,\Delta t}_{t_i},e_j)\big|^2\nu(de)dt\bigg] \nonumber\\ 
	   &\leq C\Big((1+|x|^2)\Delta t+\mathcal{R}^2_{Z^\epsilon}(\Delta t)+\mathcal{R}^2_{U^\epsilon}(\Delta t,h)+\zeta\mathcal{R}^2_{L^\epsilon}(\Delta t) +\mathcal{R}^2_\gamma(h)\label{strong error-U}\\ &\qquad\qquad\qquad\qquad\quad +\displaystyle\sum_{i=0}^{N-1}(\epsilon_i^{\mathcal{N},v}+\Delta t_i \epsilon_i^{\mathcal{N},z}+\zeta\Delta t_i\epsilon_i^{\mathcal{N},l}+\Delta t_i\epsilon_i^{\mathcal{N},\rho})\Big).\nonumber
		\end{align}
		\end{remark}
		
Combining Theorem \ref{estim-weak} and \ref{theo-convergence1}, we find the total error in the DeepBSDE approximation of the solution $u$ of the PIDE \eqref{PIDE}.
	\begin{corollary}
	(\textbf{Total errors of PIDE (\ref{PIDE})})
		Assume $\epsilon,\Delta t,h>0$,  $\zeta\in\{0,1\}$, (\textbf{H1}), (\textbf{H2}), (\textbf{H3}), and $u$ and $(X^{\epsilon,\Delta t},\hat{\mathcal{Y}},\hat{\mathcal{Z}},\hat{\mathcal{U}})$ solve \eqref{PIDE} and Algorithm \ref{ALGO1}. Then there is $C>0$ independent of $\zeta$, $\Delta t$, and $h$, such that
		\begin{align*}
		\displaystyle&\big|u(0,x)-\hat{\mathcal{Y}}_{0}(x)\big|\leq C(1+|x|)\Big(\Delta t+\zeta\epsilon^{1-\alpha/3}+(1-\zeta)\epsilon^{1-\alpha/2}\Big)\\[0.2cm] 
&+C\Big(\mathcal{R}^2_\gamma(h)+\mathcal{R}^2_{Z^\epsilon}(\Delta t)+\zeta\mathcal{R}^2_{L^\epsilon}(\Delta t)+\mathcal{R}^2_{U^\epsilon}(\Delta t,h) +\displaystyle\sum_{i=0}^{N-1}(\epsilon_i^{\mathcal{N},v}+\Delta t\epsilon_i^{\mathcal{N},z}+\zeta\Delta t\epsilon_i^{\mathcal{N},l}+\Delta t\epsilon_i^{\mathcal{N},\rho})\Big)^{\frac12}.\end{align*}
	\end{corollary}
 
Note that the rate is not better than $1/2$ in $\Delta t$ and that better rates in $\epsilon$ are obtained using approximation of small jumps by Brownian motion ($\zeta=1$).
	\section{Proof of convergence, Theorem \ref{theo-convergence1}.}\label{sec:pf}

	First we relate solutions of our DeepBSDE scheme to solutions of the intermediate scheme. 
		\begin{lemma}\label{lem-errNN} Assume $\epsilon,\Delta t,h>0$, $\zeta\in\{0,1\}$, (\textbf{H1}), (\textbf{H2}), (\textbf{H3}), and  $(X^{\epsilon,\Delta t},\hat{\mathcal{Y}},\hat{\mathcal{Z}},\hat{\mathcal{W}},\hat{\mathcal{U}})$ and $(\hat{\mathcal{V}},\bar{\hat{Z}},\bar{\hat{L}},\bar{\hat{U}})$ solve Algorithm \ref{ALGO1} and \eqref{NewBack}. Then
		there is $C>0$ independent of $\Delta t, h,\epsilon,\zeta$, such that
	\begin{equation}\label{errNN}
	\begin{array}{l}	\mathbb{E}\big|\hat{\mathcal{V}}_{t_i}-\hat{\mathcal{Y}}_i(X^{\epsilon,\Delta t}_{t_i})\big|^2+\Delta t\mathbb{E}\big|\bar{\hat{Z}}_{t_i}-\hat{\mathcal{Z}}_i(X^{\epsilon,\Delta t}_{t_i})\big|^2+\zeta\Delta t\mathbb{E}\big|\bar{\hat{L}}_{t_i}-\hat{\mathcal{W}}_i(X^{\epsilon,\Delta t}_{t_i})\big|^2\\ \\ +\Delta t \displaystyle\sum_{j\geq 1}\mathbb{E}\big|\bar{\hat{U}}_{t_i}(e_j)-\hat{\mathcal{U}}_i(X^{\epsilon,\Delta t}_{t_i},e_j)\big|^2\nu(K_j)\leq C(\epsilon_i^{\mathcal{N},v}+\Delta t\epsilon_i^{\mathcal{N},z}+\zeta\Delta t\epsilon_i^{\mathcal{N},l}+\Delta t \epsilon_i^{\mathcal{N},\rho}).
	\end{array}
	\end{equation}
	\end{lemma}
\begin{proof}
In the functional (\ref{Minfun}) to be minimised to determine $(\hat{\mathcal{Y}},\hat{\mathcal{Z}},\hat{\mathcal{W}},\hat{\mathcal{U}})$, we replace the $g$-term  and introduce $(\hat{\mathcal{V}},\bar{\hat{Z}},\bar{\hat{L}},\bar{\hat{U}})$ using equation (\ref{eq-representation}):  For $i\in \{0, \dots , N- 1\}$,
\qquad\\\qquad\\

	\begin{align*}	&\mathcal{R}_i(\theta,\eta,\lambda,\alpha)=\\
	&\quad \mathbb{E}\, \bigg|\,\hat{\mathcal{V}}_{t_i}-\mathcal{Y}_i(X^{\epsilon,\Delta t}_{t_i},\theta)
	+\Big[f(\cdots,\mathcal{Y}_i(X^{\epsilon,\Delta t}_{t_i},\theta),\cdots)
	-f(\cdots,\hat{\mathcal{V}}_{t_i}(X^{\epsilon,\Delta t}_{t_i}),\cdots)\Big]\Delta t_i\\
	&\qquad+\displaystyle\sum_{l=i+1}^{N-1}\integ{t_l}{t_{l+1}}\hat{Z}_s-\hat{\mathcal{Z}}_{l}(X^{\epsilon,\Delta t}_{t_l})\;dW_s +\displaystyle\int_{t_i}^{t_{i+1}}\hat{Z}_s-\mathcal{Z}_{i}(X^{\epsilon,\Delta t}_{t_i},\eta)\;dW_{s}\\
	&\qquad+\zeta\displaystyle\sum_{l=i+1}^{N-1}\integ{t_l}{t_{l+1}}\hat{L}_s-\hat{\mathcal{W}}_{l}(X^{\epsilon,\Delta t}_{t_l})\;d\widetilde{W}_s+\zeta\displaystyle\int_{t_i}^{t_{i+1}}\hat{L}_s-\mathcal{W}_{i}(X^{\epsilon,\Delta t}_{t_i},\lambda)\;d\widetilde{W}_{s}\\ &\qquad+\displaystyle\sum_{l=i+1}^{N-1}\sum_{j\geq 1}\integ{t_l}{t_{l+1}}\int_{K_j}\hat{U}_s(e)-\hat{\mathcal{U}}_{l}(X^{\epsilon,\Delta t}_{t_l},e_j)\;\tilde{\mu}(de,ds)\\
	&\hspace{6.4cm}+\displaystyle\sum_{j\geq 1}\integ{t_i}{t_{i+1}}\int_{K_j}\hat{U}_s(e)-\mathcal{U}_{i}(X^{\epsilon,\Delta t}_{t_i},e_j,\alpha)\;\tilde{\mu}(de,ds)\,\bigg|^2.
	\end{align*}
The first 4 terms are $\mathcal{F}_{t_i}$-measurable and therefore independent of the remaining stochastic integrals. Exploiting this independence and adding and subtracting $(\bar{\hat{Z}}_{t_l},\bar{\hat{L}}_{t_l},\bar{\hat{U}}_{t_l})$, we then have
	\begin{align*} &\mathcal{R}_i(\theta,\eta,\lambda,\alpha)=\mathbb{E}\,\Big|\hat{\mathcal{V}}_{t_i}-\mathcal{Y}_i+\Big[f(\cdots,\mathcal{Y}_i,\mathcal{Z}_i,\cdots)
	-f(\cdots,\hat{\mathcal{V}}_{t_i},\bar{\hat{Z}}_{t_i},\cdots)
	\Big]\Delta t_i\Big|^2 \\
	&+\mathbb{E}\,\bigg|\displaystyle\sum_{l=i}^{N-1}\integ{t_l}{t_{l+1}}\hat{Z}_s-\bar{\hat{Z}}_{t_l}\,dW_s+ \displaystyle\sum_{l=i+1}^{N-1}(\bar{\hat{Z}}_{t_l}-\hat{\mathcal{Z}}_{l}(X^{\epsilon,\Delta t}_{t_l}))\,\Delta W_{t_l}+(\bar{\hat{Z}}_{t_i}-\mathcal{Z}_{i}(X^{\epsilon,\Delta t}_{t_i},\eta))\,\Delta W_{t_i}
\\
&+\zeta\displaystyle\sum_{l=i}^{N-1}\integ{t_l}{t_{l+1}}\hat{L}_s-\bar{\hat{L}}_{t_l}\;d \widetilde{W}_s+ \zeta\displaystyle\sum_{l=i+1}^{N-1}(\bar{\hat{L}}_{t_l}-\hat{\mathcal{W}}_{l}(X^{\epsilon,\Delta t}_{t_l}))\,\Delta \widetilde{W}_{t_l}+\zeta(\bar{\hat{L}}_{t_i}-\mathcal{W}_{i}(X^{\epsilon,\Delta t}_{t_i},\lambda))\,\Delta \widetilde{W}_{t_i}
\\
&+\displaystyle\sum_{l=i}^{N-1}\sum_{j\geq 1}\integ{t_l}{t_{l+1}}\int_{K_j}\hat{U}_s(e)-\bar{\hat{U}}_{t_l}(e_j)\,\tilde{\mu}(de,ds)+\displaystyle\sum_{l=i+1}^{N-1}\sum_{j\geq 1}(\bar{\hat{U}}_{t_l}(e_j)-\hat{\mathcal{U}}_{l}(X^{\epsilon,\Delta t}_{t_l},e_j))\,\tilde{\mu}(K_j,(t_l,t_{l+1}])\\
&\hspace{7.5cm}+\displaystyle\sum_{j\geq 1}(\bar{\hat{U}}_{t_i}(e_j)-\mathcal{U}_{i}(X^{\epsilon,\Delta t}_{t_i},e_j,\alpha))\,\tilde{\mu}(K_j,(t_i,t_{i+1}])\bigg|^2,
	\end{align*}
Note that $\bar{\hat{Z}}_{t_i }$, $\bar{\hat{L}}_{t_i }$ and $\bar{\hat{U}}_{t_i }(e_j)$ are $L^2$-projections of $\hat{Z}_t$, $\hat{L}_t$ and $\hat{U}_t(e)$ by \eqref{piZLU}, and hence cross-terms on the same time-interval like e.g. $\mathbb E[ \int_{t_l}^{t_{l+1}}(\hat{Z}_s-\bar{\hat{Z}}_{t_l }) \,dW_s\,(\bar{\hat{Z}}_{t_l}-\hat{\mathcal{Z}}_{l}(X^{\epsilon,\Delta t}_{t_l})) \Delta W_l]$ are zero. By expanding the square and using independence of stochastic increments/integrals, orthogonality of projection errors, and the Levy-It\^{o} isometry, we then find that
	\begin{align}
	\mathcal{R}_i(\theta,\eta,\lambda,\alpha)=&\ L_i(\theta,\eta,\lambda,\alpha)+\mathbb{E}\bigg[\sum_{l=i}^{N-1} \integ{t_l}{t_{l+1}}|\hat{Z}_s-\bar{\hat{Z}}_{t_l}|^2ds\bigg]+\mathbb{E}\bigg[\sum_{l=i+1}^{N-1} |\bar{\hat{Z}}_{t_l}-\hat{\mathcal{Z}}_{l}(X^{\epsilon,\Delta t}_{t_l})|^2\Delta t_l\bigg]\nonumber\\  &+\zeta\mathbb{E}\bigg[\sum_{l=i}^{N-1} \integ{t_l}{t_{l+1}}|\hat{L}_s-\bar{\hat{L}}_{t_l}|^2ds\bigg]+\zeta\mathbb{E}\bigg[\sum_{l=i+1}^{N-1} |\bar{\hat{L}}_{t_l}-\hat{\mathcal{W}}_{l}(X^{\epsilon,\Delta t}_{t_l})|^2\Delta t_l\bigg]\nonumber\\ & +\mathbb{E}\bigg[\sum_{l=i}^{N-1}\sum_{j\geq 1}\int_{t_l}^{t_{l+1}}\int_{K_j}|\hat{U}_s(e)-\bar{\hat{U}}_{t_l}(e_j)|^2\nu(de)ds\bigg]\nonumber\\&+\mathbb{E}\bigg[\sum_{l=i+1}^{N-1}\sum_{j\geq 1}\int_{t_l}^{t_{l+1}}\int_{K_j}|\bar{\hat{U}}_{t_l}(e_j)-\hat{\mathcal{U}}_{l}(X^{\epsilon,\Delta t}_{t_l},e_j)|^2\nu(K_j)\Delta t_l\bigg],\nonumber
	\end{align}
	with
	\begin{align*} L_i(\theta,\eta,\lambda,\alpha):=&\ \mathbb{E}\,\Big|\hat{\mathcal{V}}_{t_i}-\mathcal{Y}_i(X^{\epsilon,\Delta t}_{t_i},\theta)+\Big[f(\cdots,\mathcal{Y}_i,\mathcal{Z}_i,\cdots)
	-f(\cdots,\hat{\mathcal{V}}_{t_i},\bar{\hat{Z}}_{t_i},\cdots)\Big]
	\Delta t_i\Big|^2\\[0.2cm] &+\mathbb{E}\big|\bar{\hat{Z}}_{t_i}-\mathcal{Z}_{i}(X^{\epsilon,\Delta t}_{t_i},\eta)\big|^2\Delta t_{i}+\zeta\, \mathbb{E}\big|\bar{\hat{L}}_{t_i}-\mathcal{W}_{i}(X^{\epsilon,\Delta t}_{t_i},\lambda)\,\big|^2\Delta t_{i}\\[0.2cm]
	&+\sum_{j\geq 1}\mathbb{E}\big|\bar{\hat{U}}_{t_i}(e_j)-\mathcal{U}_{i}(X^{\epsilon,\Delta t}_{t_i},e_j,\alpha)\big|^2\nu(K_j)\Delta t_i.
	\end{align*}
	By expanding the square, using the Lipschitz condition on $f$ in (H2) (i),  Cauchy-Schwartz on the $U\gamma$-terms, and the bounds on $\gamma$ in  (H2) (ii), we get an upper bound on $L_i$,
	\begin{equation}
	\begin{array}{ll}\label{est1-L}
	L_i(\theta,\eta,\lambda,\alpha)\leq C\mathbb{E}\big|\hat{\mathcal{V}}_{t_i}-\mathcal{Y}_i(X^{\epsilon,\Delta t}_{t_i},\theta)\big|^2+C\Delta t_i\mathbb{E}\big|\bar{\hat{Z}}_{t_i}-\mathcal{Z}_{i}(X^{\epsilon,\Delta t}_{t_i},\eta)\big|^2\\\\\quad+\, C\zeta\Delta t_i\mathbb{E}\big|\bar{\hat{L}}_{t_i}-\mathcal{W}_{i}(X^{\epsilon,\Delta t}_{t_i},\lambda)\big|^2+C\Delta t_i \displaystyle\sum_{j\geq 1}\mathbb{E}\big|\bar{\hat{U}}_{t_i}(e_j)-\mathcal{U}_{i}(X^{\epsilon,\Delta t}_{t_i},e_j,\alpha)\big|^2\nu(K_j).
\end{array}
	\end{equation}
A lower bound can be found using the following consequence of Young's inequality, $$(a+b)^2\geq(1-r\Delta t_i)a^2+(1-\frac{1}{r\Delta t_i})b^2\geq(1-r\Delta t_i)a^2-\frac{1}{r \Delta t_i}b^2\qquad\text{for}\qquad r>0,$$ 
combined with Lipschitz continuity of $f$,  the bounds on $\gamma$,  and the choice $r = 2C_1$,
	\begin{align}
	&L_i(\theta,\eta,\lambda,\alpha)\geq (1-r\Delta t_i)\mathbb{E}|\hat{\mathcal{V}}_{t_i}-\mathcal{Y}_i|^2\nonumber\\[0.2cm]
	&\quad-\frac{C_1\Delta t_i}{r}\Big(\mathbb{E}|\hat{\mathcal{V}}_{t_i}-\mathcal{Y}_i|^2+\mathbb{E}|\bar{\hat{Z}}_{t_i}-\mathcal{Z}_{i}|^2\nonumber  +\zeta\,\mathbb{E}|\bar{\hat{L}}_{t_i}-\mathcal{W}_{i}|^2+\sum_{j\geq 1}\mathbb{E}|\bar{\hat{U}}_{t_i}(e_j)-\mathcal{U}_{i,j}|^2\nu(K_j)\Big)\nonumber\\ &
	\quad+\mathbb{E}|\bar{\hat{Z}}_{t_i}-\mathcal{Z}_{i}|^2\Delta t_{i} +\zeta\,\mathbb{E}|\bar{\hat{L}}_{t_i}-\mathcal{W}_{i}|^2\Delta t_{i} +\sum_{j\geq 1}\mathbb{E}|\bar{\hat{U}}_{t_i}(e_j)-\mathcal{U}_{i,j}|^2\nu(K_j)\Delta t_i. \nonumber\\
	&\geq (1-C\Delta t_i)\mathbb{E}|\hat{\mathcal{V}}_{t_i}-\mathcal{Y}_i|^2\label{esti2-L}\\  
	&\quad +\frac{\Delta t_i}{2}\Big(\mathbb{E}|\bar{\hat{Z}}_{t_i}-\mathcal{Z}_{i}|^2 +\zeta\mathbb{E}|\bar{\hat{L}}_{t_i}-\mathcal{W}_{i}|^2+\sum_{j\geq 1}\mathbb{E}|\bar{\hat{U}}_{t_i}(e_j)-\mathcal{U}_{i,j}|^2\nu(K_j)\Big).\nonumber
 	\end{align}	
To get to the solutions of Algorithm \ref{ALGO1}, we now take a minimum $(\theta^*_i,\eta^*_i,\lambda^*_i,\alpha^*_i)$ of $\mathcal{R}_i(\theta,\eta,\lambda,\alpha)$ so that
	$\hat{\mathcal{Y}}_i=\mathcal{Y}_i(.,\theta_i^*)$, $\hat{\mathcal{Z}}_i=\mathcal{Z}_i(.,\eta_i^*)$, $\hat{\mathcal{W}}_i=\mathcal{W}_i(.,\lambda_i^*)$, and  $\hat{\mathcal{U}}_i=\mathcal{U}_i(.,\alpha_i^*)$. Since $\mathcal{R}_i(\theta,\eta,\lambda,\alpha)$ and $L_i(\theta_i,\eta_i,\alpha_i)$ have the same minima, by the bounds for $L_i$, (\ref{est1-L}) and (\ref{esti2-L}), we then find that
	\begin{align*}
	&(1-C\Delta t_i)\mathbb{E}\big|\hat{\mathcal{V}}_{t_i}-\hat{\mathcal{Y}}_i(X^{\epsilon,\Delta t}_{t_i})\big|^2+\displaystyle\frac{\Delta t_i}{2}\mathbb{E}\big|\bar{\hat{Z}}_{t_i}-\hat{\mathcal{Z}}_i(X^{\epsilon,\Delta t}_{t_i})\big|^2 +\zeta\displaystyle\frac{\Delta t_i}{2}\mathbb{E}\big|\bar{\hat{L}}_{t_i}-\hat{\mathcal{W}}_i(X^{\epsilon,\Delta t}_{t_i})\big|^2\\[0.2cm] &+\displaystyle\frac{\Delta t_i}{2}\displaystyle\sum_{j\geq 1}\mathbb{E}\big|\bar{\hat{U}}_{t_i}(e_j)-\hat{\mathcal{U}}_i(X^{\epsilon,\Delta t}_{t_i},e_j)\big|^2\nu(K_j)\\[0.2cm] \leq &\ L_i(\theta^*_i,\eta^*_i,\lambda_i^*,\alpha^*_i)\leq L_i(\theta_i,\eta_i,\lambda_i,\alpha_i)\\[0.2cm]
	\leq &\ C\,\mathbb{E}\big|\hat{\mathcal{V}}_{t_i}-\mathcal{Y}_i(X^{\epsilon,\Delta t}_{t_i},\theta)\big|^2+C\Delta t_i\,\mathbb{E}\big|\bar{\hat{Z}}_{t_i}-\mathcal{Z}_{i}(X^{\epsilon,\Delta t}_{t_i},\eta)\big|^2+C\zeta\Delta t_i\,\mathbb{E}\big|\bar{\hat{L}}_{t_i}-\mathcal{W}_{i}(X^{\epsilon,\Delta t}_{t_i},\lambda)\big|^2\\[0.2cm]
	&+C\Delta t_i \displaystyle\sum_{j\geq 1}\mathbb{E}\big|\bar{\hat{U}}_{t_i}(e_j)-\mathcal{U}_{i}(X^{\epsilon,\Delta t}_{t_i},e_j,\alpha)\big|^2\nu(K_j).\nonumber
	\end{align*}
By the definition of the $\epsilon_i^{\mathcal{N},\cdot}$-terms in Theorem \ref{theo-convergence1} and for $\Delta t$ small enough, we conclude that
	\begin{align*}
	&\mathbb{E}\big|\hat{\mathcal{V}}_{t_i}-\hat{\mathcal{Y}}_i(X^{\epsilon,\Delta t}_{t_i})\big|^2+\Delta t_i\mathbb{E}\big|\bar{\hat{Z}}_{t_i}-\hat{\mathcal{Z}}_i(X^{\epsilon,\Delta t}_{t_i})\big|^2+\zeta\Delta t_i\mathbb{E}\big|\bar{\hat{L}}_{t_i}-\hat{\mathcal{W}}_i(X^{\epsilon,\Delta t}_{t_i})\big|^2\\[0.2cm]
	&+\Delta t_i \displaystyle\sum_{j\geq 1}\mathbb{E}\big|\bar{\hat{U}}_{t_i}(e_j)-\hat{\mathcal{U}}_i(X^{\epsilon,\Delta t}_{t_i},e_j)\big|^2\nu(K_j)\leq C(\epsilon_i^{\mathcal{N},v}+\Delta t\epsilon_i^{\mathcal{N},z}+\zeta\Delta t\epsilon_i^{\mathcal{N},l}+\Delta t \epsilon_i^{\mathcal{N},\rho}),
	\end{align*}
	and the result follows.
\end{proof}

Next we relate solutions of the intermediate scheme to solutions of the FBSDE \eqref{FSDE2}--\eqref{FBSDE2}.

\begin{lemma}\label{lem-tilda-V}
Assume $\epsilon,\Delta t,h>0$, $\zeta\in\{0,1\}$, (\textbf{H1}), (\textbf{H2}),  $(X^\epsilon_t,Y^\epsilon_{t},Z^\epsilon_{t},L_t^\epsilon,U_t^\epsilon(e))$ solve \eqref{FSDE2}--\eqref{FBSDE2},  $(\hat{\mathcal{V}},\bar{\hat{Z}},\bar{\hat{L}},\bar{\hat{U}})$ solve  \eqref{NewBack}, and $\tilde{\mathcal{V}}$ is defined by \eqref{shema2.1}. Then there is a constant $C>0$ independent of $\Delta t$, $h$, $\zeta$, $\epsilon$ such that
		\begin{align}\label{Y-estim-rate}
		\max_{i=0,\dots, N-1}\mathbb{E}|Y^\epsilon_{t_i}-\tilde{\mathcal{V}}_{t_i}|^2\leq&\; C\bigg((1+|x|^2)\Delta t+\mathcal{R}^2_\gamma(h)+\mathcal{R}^2_{Z^\epsilon}(\Delta t)+\mathcal{R}^2_{L^\epsilon}(\Delta t)+\mathcal{R}^2_{U^\epsilon}(\Delta t,h)
		 \\  &
		\hspace{3cm}+\sum_{i=0}^{N-1} (\epsilon_i^{\mathcal{N},v}+\Delta t\epsilon_i^{\mathcal{N},z}+\Delta t\epsilon_i^{\mathcal{N},l}+ \Delta t\epsilon_i^{\mathcal{N},\rho})\bigg).\nonumber
		\end{align}
	\end{lemma}
\begin{proof}
(1) \ Let $i\in\{0,\dots,N-1\}$. Subtracting (\ref{FBSDE2}) and (\ref{shema2.1}), we find that 
	\begin{align*}
	& Y^\epsilon_{t_i}-\tilde{\mathcal{V}}_{t_i}=\mathbb{E}_i[Y^\epsilon_{t_{i+1}}-\tilde{\mathcal{V}}_{t_{i+1}}] +\mathbb{E}_i\integ{t_i}{t_{i+1}}\bigg[f\Big(t,X^\epsilon_t,Y^\epsilon_t,Z^\epsilon_t,\int_{E^\epsilon}U^\epsilon_t(e)\gamma(e)\nu(dz)+\zeta D_e\gamma(0)\Sigma^{\frac{1}{2}}_\epsilon L^\epsilon_t\Big)\nonumber\\
 &\ \ -f\Big(t_i,X^{\epsilon,\Delta t}_{t_i},\hat{\mathcal{Y}}_{t_{i}}(X^{\epsilon,\Delta t}_{t_i}),\hat{\mathcal{Z}}_{t_{i}}(X^{\epsilon,\Delta t}_{t_i}),\sum_{j\geq 1}\hat{\mathcal{U}}_{t_{i}}(X^{\epsilon,\Delta t}_{t_i},e_j)\gamma_j\nu(K_j)+\zeta D_e\gamma(0)\Sigma^\frac{1}{2}_\epsilon \hat{\mathcal{W}}_{i}(X^{\epsilon,\Delta t}_{t_i})\Big)\bigg]dt\nonumber.
	\end{align*}
	The Lipschitz assumption on $f$ and  Young's inequality, $(a+b)^2\leq (1+\delta\Delta t_i)a^2+(1+\displaystyle\frac{1}{\delta\Delta t_i})b^2$ for $\delta > 0$, then give
	\begin{align}
	|Y^\epsilon_{t_i}-\tilde{\mathcal{V}}_{t_i}|^2\leq&\;(1+\delta\Delta t_i)\big|\mathbb{E}_i[Y^\epsilon_{t_{i+1}}-\tilde{\mathcal{V}}_{t_{i+1}}]\big|^2+C_f^2\Big(1+\frac{1}{\delta\Delta t_i}\Big)\bigg|\integ{t_i}{t_{i+1}}|t_i-t|^{\frac{1}{2}}\,dt \nonumber\\&
	+\mathbb{E}_i\integ{t_i}{t_{i+1}}\big|X^\epsilon_t-X^{\epsilon,\Delta t}_{t_i}\big|\,dt+\mathbb{E}_i\integ{t_i}{t_{i+1}}\big|Y^\epsilon_t-\hat{\mathcal{Y}}_{t_{i}}(X^{\epsilon,\Delta t}_{t_i})\big|\,dt\nonumber\\ & + \mathbb{E}_i\integ{t_i}{t_{i+1}}\big|Z^\epsilon_t-\hat{\mathcal{Z}}_{t_{i}}(X^{\epsilon,\Delta t}_{t_i})\big|\,dt+\big|\zeta D_e\gamma(0)\Sigma^\frac{1}{2}_\epsilon\big|\;\mathbb{E}_i\integ{t_i}{t_{i+1}}\big|L^\epsilon_t-\hat{\mathcal{W}}_{t_{i}}(X^{\epsilon,\Delta t}_{t_i})\big|\,dt\nonumber \\ 
	&\hspace{3cm}  +\mathbb{E}_i\Big[\sum_{j\geq 1}\integ{t_i}{t_{i+1}}\int_{K_j}\big|U^\epsilon_{t}(e)\gamma(e)-\hat{\mathcal{U}}_{t_{i}}(X^{\epsilon,\Delta t}_{t_i},e_j)\gamma_j\big|\,\nu(de)\,dt\Big]\bigg|^2.\nonumber
	\end{align}
	Using $\Big(\displaystyle\sum_{i=1}^{6}a_i\Big)^2\leq 6\sum_{i=1}^{6}a^2_i$,  Fubini, Cauchy-Schwartz (for conditional expectations), and the strong error bounded for the Euler approximation (\ref{estiX}), we find that
	\begin{align}\label{Y-V1}
	\mathbb{E}|Y^\epsilon_{t_i}-\tilde{\mathcal{V}}_{t_i}|^2\leq&\,(1+\delta\Delta t_i)\mathbb{E}\big|\mathbb{E}_i[Y^\epsilon_{t_{i+1}}-\tilde{\mathcal{V}}_{t_{i+1}}]\big|^2+6C_f^2\Delta t_{i}\Big(1+\frac{1}{\delta\Delta t_i}\Big)\bigg[{ C(1+|x|^2)}\Delta t^2\\&+\integ{t_i}{t_{i+1}}\mathbb{E}\big|Y^\epsilon_t-\hat{\mathcal{Y}}_{t_{i}}(X^{\epsilon,\Delta t}_{t_i})\big|^2\,dt+\integ{t_i}{t_{i+1}}\mathbb{E}\big|Z^\epsilon_t-\hat{\mathcal{Z}}_{t_{i}}(X^{\epsilon,\Delta t}_{t_i})\big|^2\,dt\nonumber \\  
	& +\big|\zeta D_e\gamma(0)\Sigma^\frac{1}{2}_\epsilon\big|^2\integ{t_i}{t_{i+1}}\mathbb{E}\big|L^\epsilon_t-\hat{\mathcal{W}}_{t_{i}}(X^{\epsilon,\Delta t}_{t_i})\big|^2\,dt\nonumber\\  &\quad+\mathbb{E}\Big[\integ{t_i}{t_{i+1}}\Big|\sum_{j\geq 1}\int_{K_j}\big|U^\epsilon_{t}(e)\gamma(e)-\hat{\mathcal{U}}_{t_{i}}(X^{\epsilon,\Delta t}_{t_i},e_j)\gamma_j\big|\,\nu(de)\Big|^2\,dt\Big]\bigg].\nonumber
 	\end{align}
 	(2) \ We now estimate the $Y^\epsilon_t-\hat{\mathcal{Y}}_{t_{i}}$ term in \eqref{Y-V1}.
	By adding and subtracting terms and using mean square continuity  (\ref{MSC-Y}) to estimate  $Y^\epsilon_t-Y^\epsilon_{t_i}$, we find that
	\begin{align}\label{MSC-Y-V-1} &\mathbb{E}\integ{t_i}{t_{i+1}}\big|Y^\epsilon_t-\hat{\mathcal{Y}}_{t_{i}}(X^{\epsilon,\Delta t}_{t_i})\big|^2\,dt \\
	&\nonumber \leq C (1+|x|^2)\Delta t^2+4\Delta t_i\,\mathbb{E}\big|Y^\epsilon_{t_i}-\tilde{\mathcal{V}}_{t_i}\big|^2+4\Delta t_i\,\mathbb{E}\big|\tilde{\mathcal{V}}_{t_i}-\hat{\mathcal{V}}_{t_i}\big|^2+4\Delta t_i\,\mathbb{E}\big|\hat{\mathcal{V}}_{t_i}-\hat{\mathcal{Y}}_{t_{i}}(X^{\epsilon,\Delta t}_{t_i})\big|^2,
	\end{align} 
	and from the definition of $\hat{\mathcal{V}}_{t_i}$, $\tilde{\mathcal{V}}_{t_i}$ in (\ref{NewBack}), (\ref{shema2.1}), Lipschitz continuity of $f$, and Cauchy-Scwartz,
	\begin{align}\nonumber
	\mathbb{E}\big|\tilde{\mathcal{V}}_{t_i}-\hat{\mathcal{V}}_{t_i}\big|^2\nonumber &\leq \, \Delta t_i^2\,\mathbb{E}\, \Big|\mathbb{E}_i\Big[f\big(t_i,X^{\epsilon,\Delta t}_{t_i},\hat{\mathcal{Y}}_{t_{i}}(X^{\epsilon,\Delta t}_{t_i}),
	\cdots)\big)-f\big(t_i,X^{\epsilon,\Delta t}_{t_i},\hat{\mathcal{V}}_{t_i},
	\cdots\big)\Big]\Big|^2\\
	&\leq 4C^2_f\Delta t_i^2\bigg(\mathbb{E}\big|\hat{\mathcal{Y}}_{t_{i}}(X^{\epsilon,\Delta t}_{t_i})-\hat{\mathcal{V}}_{t_i}\big|^2+\mathbb{E}|\hat{\mathcal{Z}}_{t_{i}}(X^{\epsilon,\Delta t}_{t_i})-\bar{\hat{Z}}_{t_i}|^2\label{MSC-Y-V-2}\\
	&\qquad\qquad\quad+\big|\zeta  D_e\gamma(0)\Sigma^\frac{1}{2}_\epsilon\big|^2\,\mathbb{E}\big|\hat{\mathcal{W}}_{t_{i}}(X^{\epsilon,\Delta t}_{t_i})-\bar{\hat{L}}_{t_i}\big|^2\nonumber\\
	&\qquad\qquad\quad+\sum_{j\geq 1}\gamma_j^2\,\nu(K_j)\,\mathbb{E}\Big[\sum_{j\geq 1}\big|\hat{\mathcal{U}}_{t_{l}}(X^{\epsilon,\Delta t}_{t_i},e_j)-\bar{\hat{U}}_{t_i}(e_j)\big|^2\nu(K_j)\Big]\bigg).\nonumber
	\end{align}
	(3) \ We estimate the $U^\epsilon_t-\hat{\mathcal{U}}_{t_{i}}$ term in \eqref{Y-V1} and find a new estimate for \eqref{Y-V1}. Adding and subtracting terms, using Cauchy-Schwartz, and the definition of $\mathcal{R}_\gamma$ in Theorem \ref{theo-convergence1},
	\begin{align}\label{errgamaa}
	& \mathbb{E}\Big[\integ{t_i}{t_{i+1}}\Big|\sum_{j\geq 1}\int_{K_j}\big|U^\epsilon_{t}(e)\gamma(e)-\hat{\mathcal{U}}_{t_{i}}(X^{\epsilon,\Delta t}_{t_i},e_j)\gamma_j\big|\nu(de)\Big|^2dt\Big]\\ 
	&\leq 2\,\mathbb{E}\integ{t_i}{t_{i+1}}\Big[\Big|\sum_{j\geq 1}\int_{K_j}\big|U^\epsilon_{t}
	||\gamma
	-\gamma_j\big|\nu(de)\Big|^2
	+\Big|\sum_{j\geq 1}\int_{K_j}\big|U^\epsilon_{t}
	-\hat{\mathcal{U}}_{t_{i}}
	\big|\gamma_j\nu(de)\Big|^2\Big] dt\nonumber\\   
	&\leq
	2\mathcal{R}^2_\gamma(h)\,\mathbb{E}\Big[\integ{t_i}{t_{i+1}}\int_{E^\epsilon}\big|U^\epsilon_{t}
	\big|^2\nu(de)dt\Big]
	+2\,\mathbb{E}\Big[\integ{t_i}{t_{i+1}}\Big|\sum_{j\geq 1}\int_{K_j}\big|U^\epsilon_{t}
	-\hat{\mathcal{U}}_{t_{i}}
	\big|\gamma_j\nu(de)\Big |^2dt\Big].\nonumber
	\end{align}
	Then, by plugging (\ref{MSC-Y-V-1}), (\ref{MSC-Y-V-2}) and (\ref{errgamaa}) into (\ref{Y-V1}), we obtain
	\begin{align}\label{Y-V}
	&\mathbb{E}|Y^\epsilon_{t_i}-\tilde{\mathcal{V}}_{t_i}|^2 \leq(1+\delta\Delta t_i)\,\mathbb{E}\big|\mathbb{E}_i[Y^\epsilon_{t_{i+1}}-\tilde{\mathcal{V}}_{t_{i+1}}]\big|^2  +6C_f^2\Delta t_{i}\Big(1+\frac{1}{\delta\Delta t_i}\Big)\bigg[C
	(1+|x|^2)\Delta t_{i}^2\\  
	&\quad +4\Delta t_i\,\mathbb{E}|Y^\epsilon_{t_i}-\tilde{\mathcal{V}}_{t_i}|^2 + \mathbb{E}\integ{t_i}{t_{i+1}}\big|Z^\epsilon_t-\hat{\mathcal{Z}}_{t_{i}}
	\big|^2dt
	\nonumber \\ 
	&\quad +\big|\zeta D_e\gamma(0)\Sigma^\frac{1}{2}_\epsilon\big|^2\,\mathbb{E}\!\integ{t_i}{t_{i+1}}\big|L^\epsilon_t-\hat{\mathcal{W}}_{t_{i}}
	\big|^2dt+2\,\mathbb{E}\integ{t_i}{t_{i+1}}\Big| \sum_{j\geq 1}\int_{K_j}\big|U^\epsilon_{t}
	-\hat{\mathcal{U}}_{t_{i}}
	\big|\gamma_j\nu(de)\Big|^2dt\nonumber\\
	&\quad +2\mathcal{R}^2_\gamma(h)\,\mathbb{E}\integ{t_i}{t_{i+1}}\int_{E^\epsilon}|U^\epsilon_{t}|^2\nu(de)dt+4\Delta t_i\,\mathbb{E}\big|\hat{\mathcal{V}}_{t_i}-\hat{\mathcal{Y}}_{t_{i}}
	\big|^2\nonumber\\ 
	&\quad +16C^2_f\Delta t_i^3\bigg(\mathbb{E}\big|\hat{\mathcal{Y}}_{t_{i}}
	-\hat{\mathcal{V}}_{t_i}\big|^2+\mathbb{E}\big|\hat{\mathcal{Z}}_{t_{i}}
	-\bar{\hat{Z}}_{t_i}\big|^2+\big|\zeta  D_e\gamma(0)\Sigma^\frac{1}{2}_\epsilon\big|^2\mathbb{E}\big|\hat{\mathcal{W}}_{t_{i}}
	-\bar{\hat{L}}_{t_i}\big|^2\nonumber\\
	&\hspace{4.5cm}+\sum_{j\geq 1}\gamma_j^2\,\nu(K_j)\,\mathbb{E}\Big[\sum_{j\geq 1}\big|\hat{\mathcal{U}}_{t_{l}}(X^{\epsilon,\Delta t}_{t_i},e_j)
	-\bar{\hat{U}}_{t_i}(e_j)\big|^2\nu(K_j)\Big]\bigg)\bigg]\nonumber.
	\end{align}
(4) \ We estimate the $Z^\epsilon_t-\hat{\mathcal{Z}}_{t_{i}}$  and $L^\epsilon_t-\hat{\mathcal{W}}_{t_{i}}$ terms in \eqref{Y-V}. By adding and subtracting terms,
\begin{equation}
	\begin{split}\label{est-of-Z} &\mathbb{E}\integ{t_i}{t_{i+1}}\big|Z^\epsilon_t-\hat{\mathcal{Z}}_{t_{i}}(X^{\epsilon,\Delta t}_{t_i})\big|^2dt\\
	&\leq3\,\mathbb{E}\integ{t_i}{t_{i+1}}\big|Z^\epsilon_t-\pi_ {[t_i,t_{i+1}]}(Z^\epsilon)\big|^2dt+3\Delta t_{i}\,\mathbb{E}\big|\pi_ {[t_i,t_{i+1}]}(Z^\epsilon)-\bar{\hat{Z}}_{t_i}\big|^2+3\Delta t_{i}\,\mathbb{E}\big|\bar{\hat{Z}}_{t_i}-\hat{\mathcal{Z}}_{t_{i}}
	\big|^2.
	\end{split}
\end{equation}
	 By Corollary \ref{cor:proj}  and the $L^2$ inner product version of the It\^{o} isometry,
 $\Delta t_i\,\pi_ {[t_i,t_{i+1}]}(Z^\epsilon)=\mathbb E_i\,\big[\Delta W_{t_i}\int_{t_i}^{t_{i+1}}Z^\epsilon_t\,dW_t\big].$ Using the BSDE (\ref{FBSDE2}) to express the $Z$-integral, independence of increments ($\mathbb E_i\,[Y^\epsilon_{t_{i}}\,\Delta W_{t_i}]=0=\mathbb E_i\,\big[\int_{t_i}^{t_{i+1}}L^\epsilon_r d\widetilde W_r\, \Delta W_i\big]$ etc.), and definition  (\ref{tilda projection}) for $\bar{\hat{Z}}_{t_i}$, we obtain
	\begin{align*}
	 &\Delta t_{i}\big(
	\pi_ {[t_i,t_{i+1}]}(Z^\epsilon)-\bar{\hat{Z}}_{t_i}\big)= \mathbb{E}_i\big[\Delta W_{t_i}(Y^\epsilon_{t_{i+1}}-\tilde{\mathcal{V}}_{t_{i+1}})\big]\nonumber\\
 &+\mathbb{E}_i\Big[\Delta W_{t_i}\integ{t_i}{t_{i+1}}f\Big(t,X^\epsilon_t,Y^\epsilon_t,Z^\epsilon_t,\int_{E^\epsilon}U^\epsilon_t(e)\gamma(e)\nu(de){+\zeta D_e\gamma(0)\Sigma^{\frac{1}{2}}_\epsilon L^\epsilon_t}\Big)dt\Big].\nonumber
	\end{align*}
	Inside the first expectation we can subtract the term $\Delta W_{t_i}\mathbb{E}_i[Y^\epsilon_{t_{i+1}}-\tilde{\mathcal{V}}_{t_{i+1}}]$ since its $\mathbb E_i$ expectation is zero. Then by
	Cauchy-Schwartz first for the conditional expectations, and then for the time integral (with integrand $1\cdot f$), followed by the relation $\mathbb E_i|\Delta W_{t_i}|^2=\Delta t_i$,
	\begin{align}\label{est1-2-of-Z}
	&\Delta t_{i}^2\,\mathbb{E}\big|\pi_ {[t_i,t_{i+1}]}(Z^\epsilon)-\bar{\hat{Z}}_{t_i}\big|^2\leq2\Delta t_i\,\mathbb{E}\big|Y^\epsilon_{t_{i+1}}-\tilde{\mathcal{V}}_{t_{i+1}}-\mathbb{E}_i[Y^\epsilon_{t_{i+1}}-\tilde{\mathcal{V}}_{t_{i+1}}]\big|^2 \\ 
	& \qquad+2\Delta t_i^2\,\mathbb{E}\Big[\integ{t_i}{t_{i+1}}\Big|f\Big(t,X^\epsilon_t,Y^\epsilon_t,Z^\epsilon_t,\int_{E^\epsilon}U^\epsilon_t
	\gamma
	\nu(de)+\zeta D_e\gamma(0)\Sigma^{\frac{1}{2}}_\epsilon L^\epsilon_t\Big)\Big|^2dt\Big].\qquad\qquad\nonumber
	\end{align}
	Expanding the square and using the law of iterated conditional expectations lead to
	\begin{align} &\mathbb{E}\big|Y^\epsilon_{t_{i+1}}-\tilde{\mathcal{V}}_{t_{i+1}}-\mathbb{E}_i[Y^\epsilon_{t_{i+1}}-\tilde{\mathcal{V}}_{t_{i+1}}]\big|^2\label{est2-of-Z}\\
	&=\, \mathbb{E}\big|Y^\epsilon_{t_{i+1}}-\tilde{\mathcal{V}}_{t_{i+1}}\big|^2+\mathbb{E}\big|\mathbb{E}_i[Y^\epsilon_{t_{i+1}}-\tilde{\mathcal{V}}_{t_{i+1}}]\big|^2 -2\mathbb{E}\Big[\mathbb{E}_i\Big[\big(Y^\epsilon_{t_{i+1}}-\tilde{\mathcal{V}}_{t_{i+1}}\big)\mathbb{E}_i[Y^\epsilon_{t_{i+1}}-\tilde{\mathcal{V}}_{t_{i+1}}]\Big]\Big]\nonumber \\
	&=\, \mathbb{E}\big|Y^\epsilon_{t_{i+1}}-\tilde{\mathcal{V}}_{t_{i+1}}\big|^2-\mathbb{E}\big|\mathbb{E}_i[Y^\epsilon_{t_{i+1}}-\tilde{\mathcal{V}}_{t_{i+1}}]\big|^2.\hspace{4cm}\nonumber
\end{align}
The same arguments will lead to similar estimates for the $L^\epsilon_t-\hat{\mathcal{W}}_{t_{i}}$ term in \eqref{Y-V1}. In particular,
	\begin{align}\label{est2-of-L}
&\Delta t_i^2\,\mathbb{E}\big|\pi_ {[t_i,t_{i+1}]}(L^\epsilon)-\bar{\hat{L}}_{t_i}\big|^2\leq2\Delta t_i\Big(\mathbb{E}\big|Y^\epsilon_{t_{i+1}}-\tilde{\mathcal{V}}_{t_{i+1}}\big|^2-\mathbb{E}\big|\mathbb{E}_i[Y^\epsilon_{t_{i+1}}-\tilde{\mathcal{V}}_{t_{i+1}}]\big|^2\Big) \\ 
& \qquad +2\Delta t_i^2\,\mathbb{E}\Big[\integ{t_i}{t_{i+1}}\Big|f\Big(t,X^\epsilon_t,Y^\epsilon_t,Z^\epsilon_t,\int_{E^\epsilon}U^\epsilon_t(e)\gamma(e)\nu(de)+\zeta D_e\gamma(0)\Sigma^{\frac{1}{2}}_\epsilon L^\epsilon_t\Big)\Big|^2dt\Big].\nonumber
	\end{align}
(5) \ We estimate the $U^\epsilon_{t}(e)-\hat{\mathcal{U}}_{t_{i}}$  term in \eqref{Y-V}. Adding and subtracting terms and forming approximate $L^2(\nu)$ integrals using Cauchy-Schwarz, $(\int_{K_j}[\cdots]\gamma_j\,\nu(de))^2\leq \int_{K_j}\gamma_j^2\,\nu(de) \int_{K_j}[\cdots]^2\,\nu(de)$, we find that
	\begin{equation}\label{est-of-U}
	\begin{split}	&\mathbb{E}\Big[\displaystyle\integ{t_i}{t_{i+1}}\bigg|\sum_{j\geq 1}\int_{K_j}\big|U^\epsilon_{t}(e)-\hat{\mathcal{U}}_{t_{i}}(X^{\epsilon,\Delta t}_{t_i},e_j)\big|\gamma_j\nu(de)\bigg|^2dt\Big]\qquad\qquad\qquad\\
	&\leq3\ \displaystyle\sum_{j\geq 1}\gamma_j^2\nu(K_j)\ \mathbb{E}\Big[\displaystyle\integ{t_i}{t_{i+1}}\sum_{j\geq 1}\int_{K_j}\big|U^\epsilon_{t}(e)-\pi_ {[t_i,t_{i+1}],K_j}(U^\epsilon)\big|^2\nu(de)dt\Big]\\ 
	&\quad+3\ \Delta t_i\,\mathbb{E}\Big|\displaystyle\sum_{j\geq 1}\big|\pi_ {[t_i,t_{i+1}],K_j}(U^\epsilon)-\bar{\hat{U}}_{t_i}(e_j)\big|\gamma_j\nu(K_j)\Big|^2\\ 
	&\quad+3\ \Delta t_i\sum_{j\geq 1}\gamma_j^2\nu(K_j)\ \mathbb{E}\Big[\displaystyle\sum_{j\geq 1}\big|\bar{\hat{U}}_{t_i}(e_j)-\hat{\mathcal{U}}_{t_{i}}(X^{\epsilon,\Delta t}_{t_i},e_j)\big|^2\nu(K_j)\Big].
	\end{split}
	\end{equation}
	 In view of Corollary \ref{cor:proj} and the $L^2$ inner product version of the Levy-It\^{o} isometry: 
 	\begin{align*} 
  &\Delta t_i\,\nu(K_j)\,\pi_ {[t_i,t_{i+1}],K_j}(U^\epsilon)=\mathbb E_i \int_{t_i}^{t_{i+1}}\!\!\!\int_{K_j} 1\cdot U^\epsilon\,\nu(de)\,dt\\
	&=\mathbb E_i\Big[\int_{t_i}^{t_{i+1}}\!\!\!\int_{K_j} 1\,\tilde\mu(de,dt)\int_{t_i}^{t_{i+1}}\!\!\!\int_{K_j} U^\epsilon\,\tilde\mu(de,dt)\Big]=\mathbb E_i\,\Big[\tilde{\mu}(K_j,[t_i,t_{i+1}))\int_{t_i}^{t_{i+1}}\!\!\!\int_{K_j} U^\epsilon\,\tilde\mu(de,dt)\Big].
	\end{align*} 
 Using the BSDE (\ref{FBSDE2}) to express the $U$-integral, independence of increments 
 $$\mathbb E_i\,\Big[\tilde{\mu}(K_j,[t_i,t_{i+1}))\,Y^\epsilon_{t_{i}}\Big]=0=\mathbb E_i\,\Big[\tilde{\mu}(K_j,[t_i,t_{i+1}))\int_{t_i}^{t_{i+1}}L^\epsilon_r d\widetilde W_r\Big]$$ 
 etc., and definition (\ref{tilda projection}) for $\bar{\hat{U}}_{t_i}$, we obtain 
	\begin{align}\label{U-esti-1}
	& \Delta t_i\,\nu(K_j)\Big(\pi_ {[t_i,t_{i+1}],K_j}(U^\epsilon)-\bar{\hat{U}}_{t_i}(e_j)\Big)\nonumber\\
	&= \mathbb{E}_i\Big[ \tilde{\mu}(K_j,[t_i,t_{i+1}))\,(Y^\epsilon_{t_{i+1}}-\tilde{\mathcal{V}}_{t_{i+1}})\Big]\\
	&\quad +\mathbb{E}_i\Big[\tilde{\mu}(K_j,[t_i,t_{i+1}))\integ{t_i}{t_{i+1}}f\Big(t,X^\epsilon_t,Y^\epsilon_t,Z^\epsilon_t,\int_{E^\epsilon}U^\epsilon_t(e)\gamma(e)\nu(de)+\zeta D_e\gamma(0)\Sigma^{\frac{1}{2}}_\epsilon L^\epsilon_t\Big)dt\Big]\nonumber.
	\end{align}
	Again we may subtract the term $\mathbb{E}_i[Y^\epsilon_{t_{i+1}}-\tilde{\mathcal{V}}_{t_{i+1}}]$ inside the first expectation and treat this term as in part (5).
	Taking absolute values, multiplying by $\gamma_j$, and summing over $j$, followed by squaring and using Cauchy-Schwartz first for the  conditional expectations, and then for the time integral (with integrand $1\cdot f$), 
	we obtain
	\begin{align*}
	& \Delta t_i^2\,\mathbb{E}\,\Big|\displaystyle\sum_{j\geq 1}\big|\pi_ {[t_i,t_{i+1}],K_j}(U^\epsilon)-\bar{\hat{U}}_{t_i}(e_j)\big|\gamma_j\nu(K_j)\Big|^2\nonumber\\
	&\leq 2\,\mathbb{E}\Big|\displaystyle\sum_{j\geq 1}\gamma_j\tilde{\mu}(K_j,[t_i,t_{i+1}))\Big|^2\Big(\mathbb{E}\big|Y^\epsilon_{t_{i+1}}-\tilde{\mathcal{V}}_{t_{i+1}}\big|^2-\mathbb{E}\big|\mathbb{E}_i[Y^\epsilon_{t_{i+1}}-\tilde{\mathcal{V}}_{t_{i+1}}]\big|^2\Big)\\ 
	& \quad+2\,\mathbb{E}\,\Big|\sum_{j\geq 1}\gamma_j\tilde{\mu}(K_j,[t_i,t_{i+1}))\Big|^2\cdot \,\Delta t_i\,\mathbb{E}\Big[\integ{t_i}{t_{i+1}}\Big|f\Big(t,X^\epsilon_t,Y^\epsilon_t,Z^\epsilon_t,\cdots
	\Big)\Big|^2dt\Big].
	\end{align*}
	Using the Levy-It\^{o} isometry, $\mathbb{E}\Big|\displaystyle\sum_{j\geq 1}\gamma_j\tilde{\mu}(K_j,[t_i,t_{i+1}))\Big|^2=\displaystyle\sum_{j\geq 1}\gamma_j^2\,\nu(K_j)\Delta t_i$, we then find that
	\begin{equation}
	\begin{split}
	\label{esti2-of-U}
	&\Delta t_i\,\mathbb{E}\,\Big|\displaystyle\sum_{j\geq 1}\big|\pi_ {[t_i,t_{i+1}],K_j}(U^\epsilon)-\bar{\hat{U}}_{t_i}(e_j)\big|\gamma_j\nu(K_j)\Big|^2\\ 
	&\leq 2\,\sum_{j\geq 1}\gamma_j^2\,\nu(K_j)\ \Big(\mathbb{E}\big|Y^\epsilon_{t_{i+1}}-\tilde{\mathcal{V}}_{t_{i+1}}\big|^2-\mathbb{E}\big|\mathbb{E}_i[Y^\epsilon_{t_{i+1}}-\tilde{\mathcal{V}}_{t_{i+1}}]\big|^2\Big)\\ 
	& \quad+2\Delta t_i\,\sum_{j\geq 1}\gamma_j^2\,\nu(K_j)\ \mathbb{E}\Big[\integ{t_i}{t_{i+1}}\Big|f\Big(t,X^\epsilon_t,Y^\epsilon_t,
	Z^\epsilon_t,\cdots 
\Big)\Big|^2dt\Big].
	\end{split}
	\end{equation}
	(6) \ The conclusion. We now plug (\ref{est-of-Z}) --  (\ref{esti2-of-U}) into (\ref{Y-V}) with a suitable choice of $\delta$ to find that there is a constant $K$ such that
\begin{align*}
	\mathbb{E}\big|Y^\epsilon_{t_i}-\tilde{\mathcal{V}}_{t_i}\big|^2
	\leq &\ (1+\delta\Delta t)\,\mathbb{E}\big|\mathbb{E}_i[Y^\epsilon_{t_{i+1}}-\tilde{\mathcal{V}}_{t_{i+1}}]\big|^2\\
	&+K\Big(\Delta t+\frac1\delta\Big)\Big(\mathbb{E}\big|Y^\epsilon_{t_{i+1}}-\tilde{\mathcal{V}}_{t_{i+1}}\big|^2-\mathbb{E}\big|\mathbb{E}_i[Y^\epsilon_{t_{i+1}}-\tilde{\mathcal{V}}_{t_{i+1}}]\big|^2\Big) +[\text{other terms}]\\
	 \leq  & \ (1+K\Delta t)\,\mathbb{E}\big|Y^\epsilon_{t_{i+1}}-\tilde{\mathcal{V}}_{t_{i+1}}\big|^2+[\text{other terms}] \qquad \text{when}\qquad \delta\leq K. 
\end{align*}
Taking $\delta = K= 36 C_f^2\big(1+|\zeta D_e\gamma(0)\Sigma^\frac{1}{2}_\epsilon|^2+\sum_{j\geq 1}\gamma_j^2\,\nu(K_j)\big)$ is sufficient, and leads to
	\begin{align*}
	& \mathbb{E}\big|Y^\epsilon_{t_i}-\tilde{\mathcal{V}}_{t_i}\big|^2
	\leq (1+K\Delta t)\,\mathbb{E}\big|Y^\epsilon_{t_{i+1}}-\tilde{\mathcal{V}}_{t_{i+1}}\big|^2+C(1+|x|^2)\Delta t^2+C\Delta t\,\mathbb{E}\big|Y^\epsilon_{t_{i}}-\tilde{\mathcal{V}}_{t_i}\big|^2\\
	&\quad+C\,\mathbb{E}\Big[\integ{t_i}{t_{i+1}}\big|Z^\epsilon_t-\pi_ {[t_i,t_{i+1}]}(Z^\epsilon)\big|^2dt\Big]+C\zeta\,\mathbb{E}\Big[\integ{t_i}{t_{i+1}}\big|L^\epsilon_t-\pi_ {[t_i,t_{i+1}]}(L^\epsilon)\big|^2dt\Big]\nonumber\\
	&\quad 
	+C\,\mathbb{E}\Big[\integ{t_i}{t_{i+1}}\sum_{j\geq 1}\int_{K_j}\big|U^\epsilon_{t}
	-\pi_ {[t_i,t_{i+1}],K_j}(U^\epsilon)\big|^2\nu(de)dt\Big]\nonumber\\
	&\quad+C\mathcal{R}^2_\gamma(h)\,\mathbb{E}\Big[\integ{t_i}{t_{i+1}}\int_{E^\epsilon}\big|U^\epsilon_{t}
	\big|^2\nu(de)dt\Big]+C\Delta t\,\mathbb{E}\big|\hat{\mathcal{V}}_{t_i}-\hat{\mathcal{Y}}_{t_{i}}
	\big|^2\nonumber\\
	&\quad+C\Delta t\,\Big(\mathbb{E}\big|\bar{\hat{Z}}_{t_i}-\hat{\mathcal{Z}}_{t_{i}}
	\big|^2+\zeta\,\mathbb{E}\big|\bar{\hat{L}}_{t_i}-\hat{\mathcal{W}}_{t_{i}}
	\big|^2
	+\mathbb{E}\Big[\sum_{j\geq 1}\big|\bar{\hat{U}}_{t_i}(e_j)-\hat{\mathcal{U}}_{t_{i}}(X^{\epsilon,\Delta t}_{t_i},e_j)\big|^2\nu(K_j)\Big]\Big)\nonumber \\ 
	&
	\quad+C\Delta t\,\mathbb{E}\Big[\integ{t_i}{t_{i+1}}\Big|f\Big(t,X^\epsilon_t,Y^\epsilon_t,Z^\epsilon_t,\int_{E^\epsilon}U^\epsilon_t(e)\gamma(e)\nu(de)+\zeta D_e\gamma(0)\Sigma^{\frac{1}{2}}_\epsilon L^\epsilon_t\Big)\Big|^2dt\Big]\nonumber.
	\end{align*}
	From the linear growth condition on $f$ from (H1) and the second moment bounds on $(X^\epsilon_t,Y^\epsilon_t,Z^\epsilon_t,U^\epsilon_t)$ in Theorem \ref{WP_FBSDE} (which are clearly independent of $\epsilon$, see e.g.  Lemma \ref{WPeps}), we also see that
	\begin{equation}\label{condi-f}
	\sup_{\epsilon\in(0,1)}\mathbb{E}\Big[\integ{0}{T}\Big|f\Big(t,X^\epsilon_t,Y^\epsilon_t,Z^\epsilon_t,\int_{E^\epsilon}U^\epsilon_t(e)\gamma(e)\nu(de)+\zeta D_e\gamma(0)\Sigma^{\frac{1}{2}}_\epsilon L^\epsilon_s\Big)\Big|^2
	dt\Big]<\infty.
	\end{equation}
	Then by Lemma \ref{lem-errNN} and the definitions of the $\mathcal{R}^2_\cdot$ and $\epsilon_i^\cdot$ terms (see statement of Theorem \ref{theo-convergence1}), 
	\begin{align*}
    \mathbb{E}\big|Y^\epsilon_{t_i}-\tilde{\mathcal{V}}_{t_i}\big|^2
	\leq & \ (1+K\Delta t)\,\mathbb{E}\big|Y^\epsilon_{t_{i+1}}-\tilde{\mathcal{V}}_{t_{i+1}}\big|^2\\
	&+C\Delta t\Big((1+|x|^2)\Delta t
	+\mathcal{R}^2_\gamma(h)+\mathcal{R}^2_{Z^\epsilon}(\Delta t)+ \zeta\mathcal{R}^2_{L^\epsilon}(\Delta t)+\mathcal{R}^2_{U^\epsilon}(\Delta t,h)\Big)
\nonumber\\ & +C\Delta t\sum_{i=0}^{N-1}\Big(\epsilon_i^{\mathcal{N},v}+\Delta t_i\epsilon_i^{\mathcal{N},z}+\zeta\Delta t_i\epsilon_i^{\mathcal{N},l}+\Delta t_i\epsilon_i^{\mathcal{N},\rho}\Big).
\end{align*}
	By the discrete Gr\"onwall Lemma and since $g$ is Lipschitz,  we then conclude that \eqref{Y-estim-rate} holds and the proof  of Lemma \ref{lem-tilda-V}  is complete.
\end{proof}

\begin{proof} [Proof of Theorem \ref{theo-convergence1}] The error bound for the  $Y$-terms follows as a direct consequence of Lemma \ref{lem-errNN} and \ref{lem-tilda-V}. The remaining bounds follows by Lemma \ref{lem-errNN} and computations similar to those in the proof of Lemma \ref{lem-tilda-V}.  Since $|Z^\epsilon_t-\hat{\mathcal{Z}_i}(X^{\epsilon,\Delta t}_{t_i})|\leq |Z^\epsilon_t-\bar{\hat{Z}}_{t_i}|+|\bar{\hat{Z}}_{t_i}-\hat{\mathcal{Z}_i}(X^{\epsilon,\Delta t}_{t_i})|$ and similar for $L$ and $U$, and the second terms are controlled by Lemma \ref{lem-errNN}, to finish the proof we must estimate first terms,
\begin{align*}
	I := &\ \mathbb{E}\bigg[\displaystyle\sum_{i=0}^{N-1}\integ{t_i}{t_{i+1}}|Z^\epsilon_t-\bar{\hat{Z}}_{t_i}|^2dt\bigg]+\zeta\,\mathbb{E}\bigg[\displaystyle\sum_{i=0}^{N-1}\integ{t_i}{t_{i+1}}|L^\epsilon_t-\bar{\hat{L}}_{t_i}|^2dt\bigg]\\
&+\mathbb{E}\bigg[\displaystyle\sum_{i=0}^{N-1}\integ{t_i}{t_{i+1}}\Big|\displaystyle\sum_{j\geq 1}\int_{K_j}|U^\epsilon_{t}(e)-\bar{\hat{U}}_{t_i}(e_j)|\gamma_j\nu(de)\Big|^2dt\bigg]. \nonumber 
\end{align*}
By adding and subtracting projections ($Z^\epsilon_t-\bar{\hat{Z}}_{t_i}=(Z^\epsilon_t-\pi_ {[t_i,t_{i+1}]}(Z^\epsilon))+(\pi_ {[t_i,t_{i+1}]}(Z^\epsilon)-\bar{\hat{Z}}_{t_i})$ etc.), the definition of $\mathcal{R}^2$, estimates
	(\ref{est1-2-of-Z}), (\ref{est2-of-L}), (\ref{esti2-of-U}), and summing over $i$,
\begin{align}\label{estim-Z-L-U-1}
	I \leq &\ C\big(\mathcal{R}^2_{Z^\epsilon}(\Delta t)+\zeta\mathcal{R}^2_{L^\epsilon}(\Delta t)+\mathcal{R}^2_{U^\epsilon}(\Delta t,h)\big)\\
 &+4\Big(1+\zeta+\displaystyle\sum_{j\geq 1}\gamma_j^2\nu(K_j)\Big)\displaystyle\sum_{i=0}^{N-1}\Big(\mathbb{E}|Y^\epsilon_{t_{i+1}}-\tilde{\mathcal{V}}_{i+1}|^2-\mathbb{E}\big|\,\mathbb{E}_i[Y^\epsilon_{t_{i+1}}-\tilde{\mathcal{V}}_{i+1}]\big|^2\Big).\nonumber\\
 &+4\Delta t\Big(1+\zeta+\displaystyle\sum_{j\geq 1}\gamma_j^2\nu(K_j)\Big)\,\nonumber\\ 
	& \quad\cdot\mathbb{E}\bigg[\displaystyle\sum_{i=0}^{N-1}\integ{t_i}{t_{i+1}}\Big|f\Big(t,X^\epsilon_t,Y^\epsilon_t,Z^\epsilon_t,\int_{E^\epsilon}U^\epsilon_t(e)\gamma(e)\nu(de)+\zeta D_e\gamma(0)\Sigma^{\frac{1}{2}}_\epsilon L^\epsilon_t\Big)\Big|^2dt\bigg]\nonumber\\
 =&\ I_1+I_2+I_3.\nonumber
	\end{align}
We first estimate $I_3$. Since $f$ is Lipschitz (H2), $(X,Y,U,Z)$ belong to $L^2$ (Lemma \ref{WPeps} and \ref{u=Y}), and $\sum_{j\geq 1}\gamma_j^2\nu(K_j)\leq C\int |\gamma(e)|^2 \nu(de)<\infty$ by (D2) and (H2) (ii),
$$I_3\leq C(1+|x|^2)\Delta t.$$
We now estimate $I_2$. By relabeling indices for the first term in the sum below,  we have 
	\begin{align}\label{chang-indi-equa}
	  & \displaystyle\sum_{i=0}^{N-1}\Big(\mathbb{E}|Y^\epsilon_{t_{i+1}}-\tilde{\mathcal{V}}_{i+1}|^2-\mathbb{E}\big|\,\mathbb{E}_i[Y^\epsilon_{t_{i+1}}-\tilde{\mathcal{V}}_{i+1}]\big|^2\Big)\\&\leq\mathbb{E}|g(X^\epsilon_T)-g(X^{\epsilon,\Delta t}_T)|^2+\displaystyle\sum_{i=0}^{N-1}\Big(\mathbb{E}|Y^\epsilon_{t_{i}}-\tilde{\mathcal{V}}_{t_i}|^2-\mathbb{E}\big|\,\mathbb{E}_i[Y^\epsilon_{t_{i+1}}-\tilde{\mathcal{V}}_{i+1}]\big|^2\Big)=I_{21}+I_{22}.\nonumber
	\end{align}
 By (H2)(iii) and Lemma \ref{estiX}, $I_{21}\leq C(1+|x|^2)\Delta t$. To estimate $I_{22}$ we use (\ref{Y-V}) and 
add and subtract $(\bar{\hat{Z}},\bar{\hat{L}},\bar{\hat{U}})$,
\begin{align}
	&\mathbb{E}|Y^\epsilon_{t_{i}}-\tilde{\mathcal{V}}_{t_i}|^2-\mathbb{E}\big|\,\mathbb{E}_i[Y^\epsilon_{t_{i+1}}-\tilde{\mathcal{V}}_{i+1}]\big|^2\\
 & \leq \delta\Delta t\,\mathbb{E}\big|\,\mathbb{E}_i[Y^\epsilon_{t_{i+1}}-\tilde{\mathcal{V}}_{i+1}]\big|^2+6C_f^2\Delta t_{i}\big(1+\frac{1}{\delta\Delta t_i}\big)\cdot\nonumber\\
 & \quad \cdot\bigg[{ C(1+|x|^2)}
	\Delta t_{i}^2 +4\Delta t_i\,\mathbb{E}|Y^\epsilon_{t_i}-\tilde{\mathcal{V}}_{t_i}|^2+ 2\mathbb{E}\integ{t_i}{t_{i+1}}|Z^\epsilon_t-\bar{\hat{Z}}_{t_{i}}|^2dt+2\Delta t _i \,\mathbb{E}|\bar{\hat{Z}}_{t_{i}}-\hat{\mathcal{Z}}_{t_{i}}
	|^2\nonumber\\ 
 &\qquad+2\zeta|  D_e\gamma(0)\Sigma^\frac{1}{2}_\epsilon|^2\,\mathbb{E}\integ{t_i}{t_{i+1}}|L^\epsilon_t-\bar{\hat{L}}_{t_{i}}|^2dt+2\zeta|  D_e\gamma(0)\Sigma^\frac{1}{2}_\epsilon|^2\Delta t _i \,\mathbb{E}|\bar{\hat{L}}_{t_{i}}-\hat{\mathcal{W}}_{t_{i}}
	|^2\nonumber\\ & \qquad+2\mathcal{R}^2_\gamma(h)\,\mathbb{E}\integ{t_i}{t_{i+1}}\int_{E^\epsilon}|U^\epsilon_{t}(e)|^2\nu(de)dt+4\Delta t_i\,\mathbb{E}|\hat{\mathcal{V}}_{t_i}-\hat{\mathcal{Y}}_{t_{i}}(X^{\epsilon,\Delta t}_{t_i})|^2\nonumber\\ &\qquad+4\mathbb{E}\integ{t_i}{t_{i+1}}\Big|\displaystyle\sum_{j\geq 1}\int_{K_j}|U^\epsilon_{t}(e)-\bar{\hat{U}}_{t_{i}}(e_j)|\gamma_j\nu(de)\Big  |^2dt\nonumber\\ 
 &\qquad+4\Delta t_i\displaystyle\sum_{j\geq 1} \gamma_j^2\nu(K_j)\,\mathbb{E}\Big[\displaystyle\sum_{j\geq 1}|\bar{\hat{U}}_{t_{i}}(e_j)-\hat{\mathcal{U}}_{t_{i}}(X^{\epsilon,\Delta t}_{t_i},e_j)|^2\nu(K_j)\Big]\nonumber\\
 &\qquad+16C^2_f\Delta t_i^3\bigg(\mathbb{E}|\hat{\mathcal{Y}}_{t_{i}}
	-\hat{\mathcal{V}}_{t_i}|^2+\mathbb{E}|\hat{\mathcal{Z}}_{t_{i}}
	-\bar{\hat{Z}}_{t_i}|^2+|\zeta  D_e\gamma(0)\Sigma^\frac{1}{2}_\epsilon|^2\,\mathbb{E}|\hat{\mathcal{W}}_{t_{i}}
	-\bar{\hat{L}}_{t_i}|^2\nonumber\\ &\hspace{4.5cm}+\displaystyle\sum_{j\geq 1}\gamma_j^2\nu(K_j)\,\mathbb{E}\Big[\displaystyle\sum_{j\geq 1}|\hat{\mathcal{U}}_{t_{i}}(X^{\epsilon,\Delta t}_{t_i},e_j)-\bar{\hat{U}}_{t_i}(e_j)|^2\nu(K_j)\Big]\bigg)\bigg].\nonumber
\end{align}
 
Now we use the Cauchy-Schwartz inequality and the tower property of conditional expectations to get rid of the $\mathbb{E}_i$-expectation on the right hand side and Lemma \ref{lem-errNN} to estimate all the $\hat{\mathcal{Y}}_{t_{i}}
	-\hat{\mathcal{V}}_{t_i}$, $\hat{\mathcal{Z}}_{t_{i}}
	-\bar{\hat{Z}}_{t_i}$, $\hat{\mathcal{W}}_{t_{i}}
	-\bar{\hat{L}}_{t_i}$, and $\hat{\mathcal{U}}_{t_{i}}-\bar{\hat{U}}_{t_i}$ terms. After summing over $i$ and using the definition of $I$, we then find that
	\begin{align*} 
  I_{22}\leq &\, \big(\delta+C\big(\Delta t+\frac1\delta\big)\big) T\displaystyle\max_{i=0,\dots, N-1}\mathbb{E}\big|Y^\epsilon_{t_{i}}-\tilde{\mathcal{V}}_{t_i}\big|^2+ C\big(\Delta t+\frac1\delta\big)(1+|x|^2)\Delta t+
	C_0\big(\Delta t+\frac{1}{\delta}\big) I
	\nonumber \\ 
 &+C_0\big(\Delta t+\frac1\delta\big)\mathcal R_\gamma^2(h)+C\big(\Delta t+\frac1\delta\big)\big(\epsilon_i^{\mathcal{N},v}+\Delta t\,\epsilon_i^{\mathcal{N},z}+\zeta\Delta t\,\epsilon_i^{\mathcal{N},l}+\Delta t\, \epsilon_i^{\mathcal{N},\rho}\big),
	\end{align*}
 where $C_0=12(1+\zeta |D_e\gamma(0)\Sigma_\epsilon^{\frac12}|^2)C_f^2$.

Going back to \eqref{estim-Z-L-U-1}, combining all the estimates, using Lemma \ref{lem-tilda-V}, taking $\Delta t\leq \frac1\delta$, and fixing $\delta$ large enough, we find that
	\begin{align*}
 \frac12I \leq  C\Big(&\mathcal{R}^2_{Z^\epsilon}(\Delta t)+ \zeta\mathcal{R}^2_{L^\epsilon}(\Delta t)+\mathcal{R}^2_{U^\epsilon}(\Delta t,h)+\mathcal{R}^2_\gamma(h)+{ (1+|x|^2)}\Delta t\nonumber\\
&+\sum_{i=0}^{N-1} \big(\epsilon_i^{\mathcal{N},v}+\Delta t\,\epsilon_i^{\mathcal{N},z}+\zeta\Delta t\,\epsilon_i^{\mathcal{N},l}+\Delta t\, \epsilon_i^{\mathcal{N},\rho}\big)\Big),\nonumber
\end{align*}
and the proof is complete. 
	\end{proof}
	
	\appendix 
 \section{Proof of Theorem \ref{strong convergence-app}.}\label{app:A}
	
	\subsection*{\texorpdfstring{$X$} --component:}
    	By Cauchy-Schwartz inequalities, and the Levy-It\^{o} isometry,
	\begin{align*}
	&\mathbb{E}\big[\sup_{ t\leq s\leq T}|X_s-X_s^\epsilon|^2\big]\nonumber\\
	&\leq  \integ{t}{T}\mathbb{E}\Big[|b(X_s)-b(X^\epsilon_s)|^2\Big]\,ds+\integ{t}{T}\mathbb{E}\Big[|\sigma(X_s)-\sigma(X^\epsilon_s)|^2\Big]\,ds\nonumber\\
	&\quad+\integ{t}{T}\int_{E^\epsilon}\mathbb{E}\Big[|\beta(X_s,e)-\beta(X^\epsilon_s,e)|^2\Big]\nu(de)\,ds\\
	&\quad+\mathbb{E}\Big[\integ{t}{T}\int_{|e|<\epsilon}\big|\beta(X_s,e)\big|^2\nu(de)\,ds + \zeta \integ{t}{T} \big| D_e\beta(X^\epsilon_{s^-},0)\Sigma^{\frac{1}{2}}_\epsilon\big|^2 ds\Big].
	\end{align*}
	By the Lipschitz assumption (H1) on $b$, $\sigma$ and $\beta$ and  by (\ref{Con-Dbeta}) we obtain 
	\begin{align*}
	\mathbb{E}\Big[\sup_{t\leq s\leq T}|X_s-X_s^\epsilon|\Big]\leq C\bigg[\integ{t}{T}\mathbb{E}\Big[|X_s-X^\epsilon_s|^2\Big]\,ds+(1+|x|^2)\sigma_\epsilon^2\bigg].
	\end{align*}
	Therefore by Gr\"onwall’s inequality, we obtain the result. $\hfill\Box$

\subsection*{\texorpdfstring{$(Y,Z,U,L)$}--components:}
	Let $\delta Y:=Y^\epsilon-Y$, $\delta Z:=Z^\epsilon-Z$, $\delta U:=U^\epsilon-U$, $\delta f:=$ $f(t,X^\epsilon,Y^\epsilon,Z^\epsilon,\Gamma^\epsilon+H^\epsilon)-f(t,X,Y,Z,\Gamma)$,  $\Gamma^\epsilon=\int_{E}U^\epsilon(e)\gamma(e)\,\nu(de)$, and $H^\epsilon_s=\zeta D_e\gamma(0)\Sigma^{\frac{1}{2}}_\epsilon L^\epsilon_s
 $.  By applying It\^{o}'s formula to the process $|\delta Y_t|^2$, we obtain  (note the signs and $\zeta^2=\zeta$)
	\begin{align}
	& I:=\mathbb{E}\bigg[|\delta Y_t|^2+\integ{t}{T}|\delta Z|^2_s\,ds+\zeta\integ{t}{T} |L^\epsilon|^2_s\,ds+\integ{t}{T}\int_{E^\epsilon}|\delta U_s(e)|^2\,\nu(de)\,ds\bigg] \label{def-I}\\
&=\mathbb{E}\bigg[
|g(X_T)-g(X_T^\epsilon)|^2+2\integ{t}{T}\delta Y_s\,\delta f_s\,ds\bigg]-\mathbb{E}\bigg[\integ{t}{T}\int_{|e|\leq\epsilon}|U_s(e)|^2\,\nu(de)\,ds\bigg]=:I_1+I_2.\nonumber
	\end{align}

By (\ref{relation-u-Y-1})  in Lemma \ref{u=Y} and Remark \ref{U-u nonsmooth},
the Lipschitz continuity of $u$ in Theorem \ref{thm:lip}, and the growth condition
on $\beta$, and the $L^2$-bound on $X_s$ in \eqref{moment-Y},  we have
		$$ \left|U_s(e)\right|^2=\left|u(s,X_{s^-}+\beta(X_{s^-},e))-u(s,X_{s^-})\right|^2 \leq C\Big(1+\left|x\right|^2\Big)|e|^2,$$
  and then by the definition of $\sigma_\epsilon$ in \eqref{def-Sigma} it follows that
\begin{align}\label{I2}
I_2\leq C(1+|x|^2)\sigma^2_\epsilon.
\end{align}
  
	By the Lipschitz  continuity (H2)  of $f$ and $g$ we have
	\begin{align}
  I_1\leq  C\,\mathbb{E}\bigg[&|X_T-X_T^\epsilon|^2+\integ{t}{T}|\delta Y_s|\Big(|X_s-X^\epsilon_s|+|\delta Z_s|+|\delta Y_s|+\int_{E^\epsilon}\gamma(e)|\delta U_s(e)|\,\nu(de)\Big)\,ds\nonumber\\[0.2cm]
 &
 +\integ{t}{T}|\delta Y_s|\Big(\int_{|e|\leq \epsilon}\gamma(e) |U_s(e)|\,\nu(de)+ |H^\epsilon_s| 
 \Big)\,ds\bigg].\nonumber
	\end{align}
	Using $ab\leq \alpha a^2+b^2/\alpha$,  (H1), and the estimate for the $X$-component, Cauchy-Schwartz, and the It\^{o} Lemma for $H^\epsilon$,  we then obtain
	\begin{align*}
  I_1\leq  &\ C(1+|x|^2)\sigma^2_\epsilon+C(1+\alpha)\,\mathbb{E}\integ{t}{T}|\delta Y|^2_s\,ds\\
 &+C\,\mathbb{E}\frac{1}{\alpha}\bigg[\integ{t}{T}\Big(|X_s-X^\epsilon_s|^2+|\delta Z_s|^2+\zeta|D\gamma(0)|^2\Sigma_\epsilon |L^\epsilon_s|^2\Big)\,ds\\
 &+\int_{E^\epsilon}\gamma(e)^2\,\nu(de)\integ{t}{T}\int_{E^\epsilon}|\delta U_s(e)|^2\,\nu(de)\,ds+\int_{|e|\leq \epsilon}\gamma(e)^2\,\nu(de)\integ{t}{T}\int_{|e|\leq \epsilon} |U_s(e)|^2\,\nu(de)\,ds\bigg].\nonumber
	\end{align*}
 By the assumptions on $\gamma$ in (H2), the definitions of $\sigma_\epsilon$ and $\Sigma_\epsilon$ in \eqref{def-Sigma}, and the estimate on $I_2$,
$$I_1\leq C(1+\tfrac1\alpha) (1+|x|^2)\sigma^2_\epsilon + C(1+\alpha)\integ{t}{T}|\delta Y|^2_s\,ds+C\tfrac1\alpha I.$$

Combining the estimates on $I_1$ and $I_2$, and taking $\alpha$ large enough, we find that
	\begin{align*}
  I\leq  C(1+|x|^2)\sigma^2_\epsilon+C\,\mathbb{E}\integ{t}{T}|\delta Y|^2_s\,ds.
	\end{align*}
Then by the definition of $I$ and Gronwall's lemma,
	$
	\mathbb{E}|\delta Y_t|^2\leq C(1+|x|^2)\sigma^2_\epsilon,$
	and then 
	\begin{align}\label{esti-Y-epsi}I\leq C(1+|x|^2)\sigma^2_\epsilon.
	\end{align}

It only	remains to estimate
	\begin{align*} I_3:=\mathbb{E}\bigg|\zeta\integ{0}{T} L^\epsilon_sd\widetilde{W}_{s}-\integ{0}{T}\int_{|e|\leq\epsilon}U_s(e)\tilde{\mu}(de,ds)\bigg|^2,
 \end{align*}
 which by the triangle inequality, Levy-It\^{o} isometry, and estimates \eqref{I2}, (\ref{esti-Y-epsi}) is bounded by
\begin{align}\label{I3}
I_3&\leq C \,\mathbb{E}\bigg[\zeta\integ{0}{T}|L^\epsilon_s|^2ds+\integ{0}{T}\int_{|e|\leq\epsilon}|U_s(e)|^2\nu(de)ds\bigg]\leq C(I+I_2)\leq C(1+|x|^2)\sigma_\epsilon^2.
	\end{align}
By \eqref{def-I}, \eqref{esti-Y-epsi}, and \eqref{I3}, the proof of Theorem \ref{strong convergence-app} is complete. $\hfill\Box$
	
\section{Proof of (\ref{strong error-U})}\label{appB}
From (\ref{U-esti-1}) and by using Cauchy-Schwarz inequalities for conditional expectations we get 
\begin{equation*}
	\begin{array}{l}
	\Delta t_i\nu(K_j)\,\mathbb{E}\displaystyle|\pi_ {[t_i,t_{i+1}],K_j}(U^\epsilon)-\bar{\hat{U}}_{t_i}(e_j)|^2\\ \leq 2\bigg(\mathbb{E}|Y^\epsilon_{t_{i+1}}-\tilde{\mathcal{V}}_{t_{i+1}}|^2-\mathbb{E}\big|\mathbb{E}_i[Y^\epsilon_{t_{i+1}}-\tilde{\mathcal{V}}_{t_{i+1}}]\big|^2\bigg)\\  \quad+2\Delta t_i\,\mathbb{E}\bigg[\integ{t_i}{t_{i+1}}\bigg|f\bigg(t,X^\epsilon_t,Y^\epsilon_t,Z^\epsilon_t,\int_{E^\epsilon}U^\epsilon_t(e)\gamma(e)\nu(de)+\zeta D_e\gamma(0)\Sigma^{\frac{1}{2}}_\epsilon L^\epsilon_t\bigg)\bigg|^2dt\bigg].
	\end{array}
	\end{equation*}
	Then multipling by $\gamma_j^2\nu(K_j)$ and summing over j, we obtain
	\begin{equation*}
		\begin{array}{l}
	\Delta t_i\,\mathbb{E}\bigg[\displaystyle\sum_{j\geq 1}|\pi_ {[t_i,t_{i+1}],K_j}(U^\epsilon)-\bar{\hat{U}}_{t_i}(e_j)|^2\gamma_j^2|\nu(K_j)|^2\bigg]\\ \leq 2\displaystyle\sum_{j\geq 1}\gamma_j^2\nu(K_j)\bigg(\mathbb{E}|Y^\epsilon_{t_{i+1}}-\tilde{\mathcal{V}}_{t_{i+1}}|^2-\mathbb{E}\big|\mathbb{E}_i[Y^\epsilon_{t_{i+1}}-\tilde{\mathcal{V}}_{t_{i+1}}]\big|^2\bigg)\\  \quad+2\Delta t_i\displaystyle\sum_{j\geq 1}\gamma_j^2\nu(K_j)\,\mathbb{E}\bigg[\integ{t_i}{t_{i+1}}\bigg|f\bigg(t,X^\epsilon_t,Y^\epsilon_t,Z^\epsilon_t,\int_{E^\epsilon}U^\epsilon_t(e)\gamma(e)\nu(de)+\zeta D_e\gamma(0)\Sigma^{\frac{1}{2}}_\epsilon L^\epsilon_t\bigg)\bigg|^2dt\bigg].
	\end{array}
	\end{equation*}
By adding and subtracting projection to $|U^\epsilon_t(e)-\bar{\hat{U}}_{t_i}(e_j)|^2$, using the last inequality, the definition of $\mathcal{R}^2$, and summing over $i$,
we obtain
	\begin{eqnarray}
	&&\mathbb{E}\bigg[\displaystyle\sum_{i=0}^{N-1}\integ{t_i}{t_{i+1}}\bigg|\displaystyle\sum_{j\geq 1}\int_{K_j}|U^\epsilon_{t}(e)-\bar{\hat{U}}_{t_i}(e_j)|^2\gamma_j^2\nu(K_j)\nu(de)dt\bigg] \nonumber \\&&\leq 4\Delta t\displaystyle\sum_{j\geq 1}\gamma_j^2\nu(K_j)\,\mathbb{E}\bigg[\displaystyle\sum_{i=0}^{N-1}\integ{t_i}{t_{i+1}}\bigg|f\bigg(t,X^\epsilon_t,Y^\epsilon_t,Z^\epsilon_t,\int_{E^\epsilon}U^\epsilon_t(e)\gamma(e)\nu(de)\nonumber\\ & &\quad+\zeta D_e\gamma(0)\Sigma^{\frac{1}{2}}_\epsilon L^\epsilon_t\bigg)\bigg|^2dt\bigg]+4\displaystyle\sum_{j\geq 1}\gamma_j^2\nu(K_j)\displaystyle\sum_{i=0}^{N-1}\bigg(\mathbb{E}|Y^\epsilon_{t_{i+1}}-\tilde{\mathcal{V}}_{i+1}|^2-\mathbb{E}\big|\mathbb{E}_i[Y^\epsilon_{t_{i+1}}-\tilde{\mathcal{V}}_{i+1}]\big|^2\bigg)\nonumber\\ & &\quad+2\mathcal{R}^2_{U^\epsilon}(\Delta t,h).\nonumber
	\end{eqnarray}
	Then by changing the indices for the last sum above, and recalling (\ref{condi-f}), we obtain
	\begin{eqnarray}\label{com-Z-U-esti1}
	&&\mathbb{E}\bigg[\displaystyle\sum_{i=0}^{N-1}\integ{t_i}{t_{i+1}}\bigg[\displaystyle\sum_{j\geq 1}\int_{K_j}|U^\epsilon_{t}(e)-\bar{\hat{U}}_{t_i}(e_j)|^2\gamma_j^2\nu(K_j)\nu(de)\bigg]dt\bigg]  \\&&\leq C\bigg(\mathcal{R}^2_{U^\epsilon}(\Delta t,h)+\Delta t\bigg)+4\displaystyle\sum_{j\geq 1}\gamma_j^2\nu(K_j)\,\mathbb{E}|g(X^\epsilon_T)-g(X^{\epsilon,\Delta t}_T)|^2\nonumber\\ & &\quad+4\bigg(1+\zeta+\displaystyle\sum_{j\geq 1}\gamma_j^2\nu(K_j)\bigg)\displaystyle\sum_{i=0}^{N-1}\bigg(\mathbb{E}|Y^\epsilon_{t_{i}}-\tilde{\mathcal{V}}_{t_i}|^2-\mathbb{E}\big|\mathbb{E}_i[Y^\epsilon_{t_{i+1}}-\tilde{\mathcal{V}}_{i+1}]\big|^2\bigg).\nonumber
	\end{eqnarray}
Now, by following the same arguments as in proof of Theorem \ref{theo-convergence1}, we obtain the result. $\Box$
	\section*{Acknowledgements}
E.R.J. received funding from the Research Council of
Norway under Grant Agreement No. 325114 “IMod. Partial differential equations,
statistics and data: An interdisciplinary approach to data-based modelling”.

\end{document}